\documentclass[a4paper,11pt,reqno]{amsart}

\usepackage{enumerate}

\usepackage{amsmath}
\usepackage{amsfonts}
\usepackage{amssymb}
\usepackage{graphicx}
\usepackage{geometry}
\usepackage{ textcomp }
\usepackage{tikz}
\usepackage[all]{xy}
\newcommand{\grau}{gray!60}
\newenvironment{grform}{\begin{tikzpicture}[intext]}{\end{tikzpicture}}
\tikzset{
norm/.style = {ultra thick, color = #1},
norm/.default = black,
intext/.style = {baseline = (current bounding box.center)},
}

\newcommand{\vLine}[5]{
  \draw[ultra thick, color = #5, rounded corners] (#1 , #2) -- (#1 , #2 * 0.7 +
#4 * 0.3 ) -- ( #3 , #2 * 0.3 + #4 * 0.7 ) -- (#3 ,#4);
}
\newcommand{\vLineO}[5]{
  \draw[line width = 6pt , color = white, rounded corners] (#1 , #2) -- (#1 , #2
* 0.7 + #4 * 0.3 ) -- ( #3 , #2 * 0.3 + #4 * 0.7 ) -- (#3 ,#4);
  \draw[ultra thick, color = #5, rounded corners] (#1 , #2) -- (#1 , #2 * 0.7 +
#4 * 0.3 ) -- ( #3 , #2 * 0.3 + #4 * 0.7 ) -- (#3 ,#4);
}

\newcommand{\dSkewantipode}[3]{
\draw[ultra thick, color = #3] (#1 , #2) -- (#1 , #2 + 1);
\filldraw[ultra thick, color = #3, fill = gray] (#1 , #2 + 0.5) circle [radius =
8pt];
}

\newcommand{\dAction}[6]{
\draw[ultra thick, color = #5] (#1 , #2) -- (#1 , #2 + #4 * 0.2) .. controls (#1
, #2 + #4 * 0.5) .. (#1 + #3 * 0.8 , #2 + #4 * 0.5) -- (#1 + #3 , #2 + #4 *
0.5);
\draw[ultra thick, color = #6] (#1 + #3 , #2) -- (#1 + #3 , #2 + #4);
}

\newcommand{\dPairing}[6]{
\draw[ultra thick, color = #5] 
(#1 , #2) -- 
(#1 , #2 + #4 - 1) .. controls (#1 , #2 + #4 - 0.5) .. (#1 + #3 * 0.48 , #2 + #4
- 0.5) -- 
(#1 + #3 * 0.52 , #2 + #4 - 0.5) .. controls (#1 + #3 , #2 + #4 - 0.5) .. (#1 +
#3 , #2 + #4 - 1) -- 
(#1 + #3 , #2);
\draw (#1 + #3 * 0.5 , #2 + #4 - 0.5) node [fill = white] {#6};
}

\newcommand{\dCopairing}[6]{
\draw[ultra thick, color = #5] 
(#1 , #2) -- 
(#1 , #2 - #4 + 1) .. controls (#1 , #2 - #4 + 0.5) .. (#1 + #3 * 0.48 , #2 - #4
+ 0.5) -- 
(#1 + #3 * 0.52 , #2 - #4 + 0.5) .. controls (#1 + #3 , #2 - #4 + 0.5) .. (#1 +
#3 , #2 - #4 + 1) -- 
(#1 + #3 , #2);
\draw (#1 + #3 * 0.5 , #2 - #4 + 0.5) node [fill = white]
{#6};
}

\theoremstyle{plain}
\newtheorem{theorem}{Theorem}
\numberwithin{theorem}{section}

\newtheorem{question}[theorem]{Question}

\newtheorem{corollary}[theorem]{Corollary}

\newtheorem{definition}[theorem]{Definition}
\newtheorem{example}[theorem]{Example}

\newtheorem{lemma}[theorem]{Lemma}

\newtheorem{fact}[theorem]{Fact}

\newtheorem{remark}[theorem]{Remark}

\newtheorem*{exampleX}{Example}

\renewcommand{\H}{\mathrm{H}}
\newcommand{\gr}{\mathrm{gr}}
\newcommand{\DG}{{DG}}
\renewcommand{\mod}{\mathrm{mod}}
\newcommand{\comod}{\mathrm{comod}}
\newcommand{\coin}{\mathrm{coin}}
\newcommand{\rank}{\mathrm{rank}}
\newcommand{\Aut}{\mathrm{Aut}}
\newcommand{\Rep}{\mathrm{Rep}}
\newcommand{\Bigal}{\mathrm{Bigal}}
\newcommand{\BrPic}{\mathrm{BrPic}}
\newcommand{\EqMon}{\mathrm{Eq_{mon}}}
\newcommand{\EqBr}{\mathrm{Eq_{br}}}
\newcommand{\EqBrN}{\mathrm{Eq_{br}^0}}
\newcommand{\AutMon}{\mathrm{Aut_{mon}}}
\newcommand{\AutBr}{\mathrm{Aut_{br}}}
\newcommand{\IsoHopf}{\mathrm{Iso_{Hopf}}}
\newcommand{\AutHopf}{\mathrm{Aut_{Hopf}}}

\newcommand{\OutHopf}{\mathrm{Out_{Hopf}}}

\newcommand{\End}{\mathrm{End}}

\newcommand{\Cent}{\mathrm{Cent}}
\newcommand{\Vect}{\mathrm{Vect}}
\newcommand{\Out}{\mathrm{Out}}

\newcommand{\id}{\mathrm{id}}

\newcommand{\sgn}{\mathrm{sgn}}
\newcommand{\D}{\mathrm{D}}
\newcommand{\DD}{\mathbb{D}}
\renewcommand{\SS}{\mathbb{S}}

\newcommand{\F}{\mathbb{F}}

\newcommand{\Z}{\mathrm{Z}}
\newcommand{\Ind}{\mathrm{Ind}}
\newcommand{\B}{\mathrm{B}}

\newcommand{\GL}{\mathrm{GL}}
\newcommand{\Sp}{\mathrm{Sp}}
\renewcommand{\sl}{\mathfrak{sl}}
\newcommand{\g}{\mathfrak{g}}

\renewcommand{\O}{\mathrm{O}}

\newcommand{\ZZ}{\mathbb{Z}}
\newcommand{\NN}{\mathbb{N}}
\newcommand{\CC}{\mathbb{C}}

\newcommand{\cat}{\mathcal{C}}

\newcommand{\dcat}{\mathcal{D}}
\newcommand{\ocat}{\mathcal{O}}
\newcommand{\ncat}{\mathcal{N}}
\newcommand{\mcat}{\mathcal{M}}

\newcommand{\pcat}{\mathcal{R}}
\newcommand{\wcat}{\mathcal{W}}
\newcommand{\zcat}{\mathcal{Z}}
\newcommand{\vcat}{\mathcal{V}}

\newcommand{\md}{\text{-}}

\newcommand{\Nic}{\mathfrak{B}}

\newcommand{\bcat}{\mathcal{BV}}
\newcommand{\ecat}{\mathcal{EV}}

\renewcommand{\D}{\mathrm{D}}

\newcommand{\hamburger}[4] 
{
  \thispagestyle{empty}
  \vspace*{-2cm}
   \begin{flushright}
     ZMP-HH #2 \\
     Hamburger Beitr\"age zur Mathematik Nr. #3 \\
     #4 \\
   \end{flushright}
  \vspace{0.5cm}
  \begin{center}
    \Large \bf
    #1
  \end{center}
  \vspace{0.5cm}
  \begin{center}        
    Simon Lentner, Jan Priel \\
    Fachbereich Mathematik, Universit\"at Hamburg \\
    Bereich Algebra und Zahlentheorie \\
    Bundesstra\ss e 55, D-20146 Hamburg \\
  \end{center}
  \vspace{0.5cm}
}

\begin{document}

\hamburger{Three natural subgroups of the Brauer-Picard group \\ of a 
Hopf algebra with applications}{17-8}{646}{February 2017}
\thispagestyle{empty}
\enlargethispage{1cm}

\begin{abstract}
In this article we construct three explicit natural subgroups of the  
Brauer-Picard group of the category of representations of a  
finite-dimensional Hopf algebra. In examples the Brauer Picard group decomposes 
into an ordered product of these subgroups, somewhat similar to a Bruhat 
decomposition.

Our construction returns for any Hopf algebra three types of braided
autoequivalences and correspondingly three families of invertible bimodule
categories. This gives examples of so-called (2-)Morita
equivalences and defects in topological field theories. We have a closer 
look at the case of quantum groups and Nichols algebras and give 
interesting applications. Finally, we briefly discuss the three families of 
group-theoretic extensions.
\end{abstract}


\makeatletter
\@setabstract
\makeatother

\section{Introduction}

For a finite tensor category $\cat$ the \emph{Brauer-Picard group}
$\BrPic(\cat)$ is defined as the group of equivalence classes of invertible 
exact $\cat$-$\cat$-bimodule categories. This group is an important invariant of 
the tensor
category $\cat$ and appears at essential places such as group-theoretic 
extension of $\cat$ and as defects in mathematical physics, 
see applications below.
By a result in \cite{ENO09}\cite{DN12} the group is isomorphic to braided 
autoequivalences of the Drinfeld center $\BrPic(\cat) \cong 
\Aut_{br}(\zcat(\cat))$; this will be crucial in what follows.\\

Computing the Brauer-Picard group, even for $\cat=\Rep(G)$ or equivalently 
$\cat=\Vect_G$ for a finite group $G$, is already an interesting and 
non-trivial task, see 
\cite{ENO09} \cite{NR14} \cite{FPSV14} \cite{LP15} \cite{MN16}. The group 
multiplication is particularly hard to pin down.
For $\cat=H\md\mod$ with $H$ an arbitrary Hopf algebra, not much is 
known besides few examples, see \cite{FMM14} \cite{Mom12} \cite{BN14} 
\cite{ZZ13}.\\

In \cite{LP15} we have proposed an approach to calculate $\BrPic(\cat)$
for $\cat=H\md\mod$ by defining certain natural subgroups\footnote{We choose 
these names $\ecat,\bcat$ for compatibilities with previous conventions. Be 
advised that $\vcat$ does {\bf not} necessarily have complement subgroups 
$\mathcal{B},\mathcal{E}$ in $\bcat,\ecat$ in the most general cases.} 
$\bcat,\ecat$ with intersection $\vcat$ and a set of elements $\pcat$, such that 
the Brauer Picard group may decompose as a Bruhat-alike decomposition 
$$\BrPic(\cat)=\bigcup_{r\in\pcat} \mathcal{BV\;EV}\;r$$
In {\it cit. loc.} we have proven such a decomposition for 
the case $H=\CC[G]$ for elements fulfilling an additional restriction 
(laziness). Moreover we checked the decomposition in all available examples by 
hand. It is unclear at this point if it is true in general.

\noindent
The intuition arises from 

\begin{exampleX}[Sec. \ref{sec_Fp_AutBr}]
Let $G\cong \ZZ_p^n$ with $p$ a prime number. Our decomposition reduces to the 
Bruhat decomposition
of $\BrPic(\Vect_G)$, which is the Lie group 
$\O_{2n}(\F_p)$ over the finite field $\F_p$. In this case $\bcat,\ecat$ are lower and upper triangular matrices, 
intersecting in the subgroup $\vcat=\GL_n(\F_p)$. The partial dualizations are 
Weyl group elements. More precisely, our result reduces to the Bruhat 
decomposition of the Lie groups $D_n$ relative to the parabolic 
subsystem $A_{n-1}$, so reflections are actually equivalence classes 
corresponding to $n+1$ cosets of the parabolic Weyl group. 
\end{exampleX}

The present article is devoted to start the discussion of the more general 
case $\cat=H\md\mod$. We shall not try to prove a decomposition theorem, but 
focus our attention on establishing and discussing the expected natural 
subgroups $\vcat,\bcat,\ecat,\langle\pcat\rangle$ of the Brauer Picard group. 
We will also briefly discuss several interesting applications of our results, 
in particular when $H$ is the Borel part of a quantum group resp.  a 
Nichols algebra.\\

\noindent
In Section \ref{sec_category} we briefly recall the induction functor and the 
ENOM-functor \cite{ENO09}
$$\AutMon(\cat)\to \BrPic(\cat)
\qquad 
\BrPic(\cat)\stackrel{\sim}{\rightarrow}\AutBr(\zcat(\cat))$$
In view of 
interesting examples and the applications to defects in mathematical physics 
and Nichols algebras we state the obvious generalization of these concepts to 
the groupoid setting, so that arbitrary monoidal 
equivalences $\cat\stackrel{\sim}{\to}\dcat$ give rise to invertible 
$\cat$-$\dcat$-bimodule categories, and these are in bijection to braided 
equivalences $\zcat(\cat)\stackrel{\sim}{\to}\zcat(\dcat)$.\\    

\noindent
In Section \ref{sec_subgroups} we define and derive for each subgroup 
$\vcat,\bcat,\ecat$ and the subset $\pcat$ explicit 
expressions for the braided autoequivalence as 
well as the invertible bimodule categories.

On one hand $\bcat$ resp. $\ecat$ are obtained using induction functors from 
$H\md\mod$ resp. $H^*\md\mod$. So the bimodule categories in 
$\BrPic(H\md\mod)$ resp. $\BrPic(H^*\md\mod)$ are given by definition. We then 
calculate explicitly the images under the ENOM functor using Bigalois objects 
and finally we describe again the preimage of $\ecat$ now in 
$\BrPic(H\md\mod)$. As linear categories, the bimodule categories in $\bcat$ are 
all equal to $\cat$, while the bimodule categories in $\ecat$ are representation 
categories of Bigalois objects, as in \cite{FMM14}.

On the other hand the set of elements $\pcat$ is defined as partial 
dualizations on the $\AutBr$-side of the functor as obtained by the first 
author in \cite{BLS15}. There are two types of partial 
dualization, for every way to decompose $H=K\rtimes A$ into a (semidirect) 
Radford-biproduct. As linear categories, the bimodule categories in $\pcat$ are 
representations of semidirect factors of $H$ (so they may be significantly 
``smaller'', down to $\Vect$) but with a largely nontrivial bimodule category 
structure $(V.M).W\stackrel{\sim}{\longrightarrow} V.(M.W)$.\\

\noindent
In Section \ref{sec_examples} we discuss examples: Mostly we work out the 
result for $\cat=\Vect_G$, which has been discussed extensively. In particular 
we discuss how our bimodule categories look in the explicit description of 
\cite{ENO09}\cite{Dav10}. Then we thoroughly discuss the case where $H$ is the 
Taft algebra and compare our results with \cite{FMM14}.\\

\noindent
In Section \ref{sec_applications} we discuss applications:
\begin{enumerate}[a)]
 \item First we discuss interesting types of bimodule categories that arise 
from our constructions for a Nichols algebra $H=\Nic(M)\rtimes \CC[G]$. 
This includes for example the quantum group   Borel parts $U_q^\geq(\g)$ 
resp. $u_q^\geq(\g)$. 
  
  First, due to the 
  Bigalois objects there are interesting elements in $\bcat,\ecat$ related to 
  different {\it liftings} of quantum groups, most of which have non-equivalent 
representation categories $\cat,\dcat,\ldots$ but are   connected by invertible 
bimodule categories. 
  
  Even more interesting are the 
  partial dualizations: We may either dualize on the Cartan part $\CC[G]$, then 
  we obtain invertible bimodule categories between different {\it forms} of 
  $u_q^\geq(\g)$ e.g. between the adjoint and the simply-connected form. 
  
  Alternatively we we may dualize on parabolic sub-Nichols algebras, then 
  partial dualization reduces to the usual Weyl group reflection of the quantum 
  group. In this way we get invertible bimodule categories connecting different 
  choices of positive roots, and as linear categories these are representations 
  of coideal subalgebras.
  
  At last, we remark that the $\Aut_{br}$-side of all these elements, which we 
  have worked out explicitly in the previous sections, give rise to braided 
  autoequivalences of the representation category of the full quantum group.  \\
  
  \item An interesting application to mathematical physics are defects:  
  (Bi-)module categories appear as boundary conditions and defects in 
  $3d$-TQFT, in particular the Brauer-Picard group is the symmetry group of 
  such theories, see \cite{FSV13},\cite{FPSV14}.
  
  Our results give three systematic, generic families of examples for 
  such defects.   More importantly, they give many examples of invertible 
  bimodule categories between different categories. In a general TQFT the 
  defects separate different regions of space, which can be labeled by 
  different categories. Particularly interesting in this matter are again the 
  concrete examples arising from quantum groups.\\
  
  \item Finally, a leading motivation for the consideration of the 
  Brauer Picard group is, that group-theoretic extensions of 
  categories are parametrized by group homomorphisms into the Brauer Picard 
  group \cite{ENO09}. We close this article by briefly discussing, which types 
  of categories arise for our three subgroups. 
  
  This includes representations 
  of the folded Nichols algebras over nonabelian groups constructed by the 
  first author in \cite{Len12}.
\end{enumerate}

\section{Categorical Setup}\label{sec_category}

\noindent
Let $\cat,\dcat,\ldots$ be finite tensor categories with base field $k=\CC$.
\begin{definition}
  The \emph{Brauer Picard Groupoid} $\BrPic$ has as objects tensor categories 
  $\cat,\dcat,\ldots$ and as morphisms equivalence classes of exact invertible 
bimodule 
  categories $_\cat\mcat_{\dcat}$ and as composition the relative Deligne 
tensor 
product $(_\cat\mcat_\dcat )\boxtimes_\dcat (_\dcat\ncat_\mathcal{E})$.\\
  The automorphism group of an object $\cat$  is the \emph{Brauer Picard group} 
  $\BrPic(\cat)$. Categories $\cat,\dcat$ for which there exists an isomorphism 
  $_\cat\mcat_{\dcat}$ are called \emph{(2-) Morita equivalent}
\end{definition}
\begin{definition}
  The \emph{monoidal equivalence groupoid} $\EqMon$ has as objects finite  
tensor categories $\cat,\dcat\ldots$ and as morphism monoidal category 
equivalences $F:\cat\to\dcat$ and as composition concatenation.\\
  The \emph{braided equivalence groupoid} $\EqBr$ has as objects braided 
tensor categories $\zcat,\wcat \ldots$ and as morphism braided category 
equivalences $F:\zcat\to\wcat$. We denote by $\EqBrN$ the full subgroupoid 
consisting of objects that are Drinfeld centers (i.e. Witt class $0$).\\
The automorphism group of an object $\cat$ is the group of \emph{monoidal 
autoequivalences} $\EqMon(\cat)=\AutMon(\cat)$ resp. \emph{braided 
autoequivalences} $\EqBr(\zcat)=\AutBr(\zcat)$.
\end{definition}
In fact we are actually dealing with a bicategory with 1-morphisms invertible 
bimodule categories and with 2-morphisms bimodule category equivalences,  
respectively with 1-morphism category equivalences and with 2-morphisms natural 
transformations.

\begin{lemma}[Induction Functor]\label{thm_induction}
  There is an evident groupoid homomorphism $\Ind:\;\EqMon\to \BrPic$ given on 
objects 
by the identity and on morphisms ${_\cat}F_{\dcat}$ by $F\mapsto  {_F}\dcat$ 
where $\dcat$ is the 
trivial right $\dcat$-module category and the trivial left $\dcat$-module 
category 
precomposed with the monoidal functor $F$.\\
This yields in particular an evident group homomorphism $\AutMon(\cat)\to 
\BrPic(\cat)$. 
\end{lemma}

\noindent
The following theorem is due to \cite{ENO09}; see \cite{DN12} for 
the non-semisimple case:
\begin{theorem}[ENOM functor]
 There is an equivalence of groupoids $\Phi:\BrPic\cong \EqBrN$.\\ 
 It is given on objects by sending $\cat\mapsto \zcat(\cat)$, on morphisms 
$_\cat\mcat_\dcat\mapsto F_\mcat$ it fulfills the 
following defining property:

$\zcat(\cat)$ acts on $_\cat\mcat_{\dcat}$ as 
bimodule 
category automorphism, where the compatibility constraint $(c.m).d\to c.(m.d)$ 
is given by the bimodule category structure and the compatibility constraint 
$c'.(c.m)\to c.(c'.m)$ is given by the half-braiding $\tau_{c,c'}$ of the 
element $(c,\tau)\in \zcat(\cat)$. Similarly $\zcat(\dcat)$ acts on 
$_\cat\mcat_{\dcat}$ as 
bimodule category automorphism. The defining property for 
$\Phi(\mcat):\zcat(\cat)\to \zcat(\dcat)$ 
is that the module category homomorphisms $c.$ and $.\Phi(\mcat)(c)$ are 
equivalent, i.e. there is a natural transformations between these two functors 
that satisfy certain coherence properties with the two module category and the 
bimodule category structure.
\end{theorem}

\section{Subgroups of $\BrPic$}\label{sec_subgroups}

\subsection{Motivation}~\\

\noindent
Why should we hope for a Bruhat-like decomposition of $\BrPic(H\md\mod)$?\\

The main motivation for our initial work \cite{LP15} was the case $H=\CC[G]$ 
for $G$ abelian, as treated in the second authors joint paper \cite{FPSV14}. In 
particular let $G\cong \ZZ_p^n$ with $p$ a prime number. Then it is known that 
$\BrPic(\Rep(G))=\Sp_{2n}(\F_2)$ resp. $=\O_{2n}(\F_p)$ and the choice of 
generators in {\it cit. loc.} are upper triangular matrices containing the 
group of group automorphisms $\Aut(G)=\GL_n(\F_p)$, and additional generators 
are the so-called {\it EM-dualities}. 

As it turned out in our study, these generators are not arbitrary, but rather 
naturally defined subgroups, in much more general context, that can be written 
down without prior knowledge of the full Brauer Picard group and come from 
different sources:

Two sets of generators can be obtained via different induction functors from 
various categories $\cat'$ with $\zcat(\cat')\cong \zcat(\cat)$, leading in the 
example for $\cat=\Vect_G$  to upper-triangular matrices $\bcat=\Aut(G)\ltimes 
H^2(G,\CC^\times)$, as in \cite{NR14}, and for $\cat=\Rep(G)$ to  
lower-triangular matrices $\ecat$ intersecting precisely in $\vcat=\Aut(G)$. 

A third set of generators, the 
so-called EM-dualities $\pcat$, turned out to be rather general braided 
autoequivalences called {\it partial dualizations} in the first authors 
work \cite{BLS15}. These can be defined whenever a Hopf algebra decomposes into 
a semidirect product, and a special case are simple reflections of quantum 
groups. \\

In \cite{LP15} we have proved that every element fulfilling an additional 
condition (laziness) decomposes accordingly into an ordered product in these 
subgroups, also we have checked the Brauer-Picard group in known cases by 
hand. The Brauer-Picard group decomposition retains roughly the properties that 
a Lie group over a ring admits (not an honest Bruhat decomposition), 
which is what we get e.g. for $G=\ZZ_k^n$ for $k$ not prime.\\

A maybe more convincing reason for our approach arose during the work on 
\cite{LP15}: Every braided autoequivalence of $\D H\md \mod$ is described through 
its action on objects plus a monoidal structure i.e. an element in $H^2(\D 
H^*,\CC^\times)$. While the action on objects seems easily accessible (one can 
look at invertible objects, stabilizer etc.), there is in general very many 
possibilities. In the lazy case this action if given by precomposing a Hopf 
algebra automorphism, and the automorphism group reminds on a matrix group, but 
for more general cases we don't have this luxury. 

On the other hand $H^2(\D H^*,\CC^\times)$ is rather technical, but it should 
not surprise us that is is connected to the groups $H^2(H,\CC^\times),\;H^2( 
H^*,\CC^\times)$ and some interaction between $H,H^*$. So we propose to shift 
classification effort to the monoidal structure of the functor, rather that its 
action on objects. In fact for abelian groups (and much more general situations) 
we have by Schauenburg \cite{Schau02} a {\it K\"unneth-type} formula, and this 
decomposition {\bf does} precisely explain the initially observed 
decomposition.\\

Another interesting question is, if one can characterize elements inside 
one Bruhat-cell: Indeed for $H=\CC[G]$ the ``big 
cell'' $\bcat\ecat$ has the property that (in the language of \cite{NR14}) it 
sends the Langrangian subcategory $\mathcal{L}_{1,1}$ to some 
$\mathcal{L}_{N,\mu}$ with $\mu$ nondegenerate. Smaller Bruhat-cells 
$\bcat\ecat r$ can be characterized by the degree of degeneracy of $\mu$, down 
to $\mu=1$ which is a pure reflection. A similar picture seems to emerge in 
this article for the bimodule categories, where the big cell consists of 
$R\md\mod$ for some algebra of same dimension as $H$, while smaller cells are 
representations of considerable smaller algebras down to $\Vect$ for the longest 
element in $\pcat$.  \\

However, these are merely speculative observations. As stated in the 
introduction, the present paper does not concern itself with the decomposition, 
but focuses solely on the definition and description of these generic subgroups 
in the general case: ~\\ 

\subsection{$\vcat$ induced from Hopf automorphisms}~\\

\noindent
This obvious subgroup reappears as the intersection of the two upcoming 
subgroups.
\newcommand{\vZ}{{^{v^{-1}}\hspace{-.4cm}}{_{v\;\;}}Z}

\begin{lemma}
 Let $v\in\IsoHopf(H,L)$ be a Hopf algebra isomorphism, then we have in 
particular a monoidal equivalence 
$v:\;L\md\mod\to H\md\mod$ by 
precomposition. Induction (Lm. \ref{thm_induction}) provides 
an invertible bimodule category $\mcat:={_v}(H\md\mod)$. \\
We claim that this element in $\BrPic(L\md\mod,H\md\mod)$ gives under the ENOM 
functor rise to the functor in $\EqBr(\D L\md\mod,\D H\md\mod)$ given on 
objects by $\Phi(\Ind(v)):Z\mapsto \vZ$ and with trivial monoidal structure. 
Similarly induction of $v^{-1}:L^*\md\mod\to H^*\md\mod$ provides a module 
category $_{v^{-1}}(H^*\md\mod)$ giving rise to the same element.\\
In particular this defines a subgroup $\vcat\subset \BrPic(H\md\mod)$ with 
$\vcat\cong \OutHopf(H)$.
\end{lemma}
\begin{proof}
 To apply the defining property of the ENOM functor it suffices to construct a 
natural isomorphism between the functors $Z.$ and $.\Phi(\Ind(v))Z$ for $M\in 
L\md\mod$: The half-braiding given by the coaction on $Z$ gives a natural 
isomorphism  of $H$-modules:    
 $${_v}Z\otimes M\to M\otimes .\vZ$$
 $$z\otimes m\mapsto v^{-1}(z^{(-1)}).m\otimes z^{(0)}$$
 We moreover have to check compatibility with the module category constraints, 
namely for all $W\in L\md\mod$ the following  equality, which requires the 
coaction choice ${^{v^{-1}}}Z$:
 $$ \begin{array}{lllllll}
 {_v}Z\otimes ({_v}W\otimes M) &\to& {_v}W\otimes ({_v}Z\otimes M)& \to& 
{_v}W\otimes (M\otimes \vZ)
 &\stackrel{=}{\to}& ({_v}W\otimes M)\otimes\vZ\\
 z\otimes w\otimes m &\mapsto& z^{(-1)}.w\otimes z^{(0)}\otimes m &\mapsto& 
  z^{(-2)}.w\otimes v^{-1}(z^{(-1)}).m\otimes z^{(0)} &&
 \end{array}$$
 $$ \begin{array}{lll}
 {_v}Z\otimes ({_v}W\otimes M) 
 &\to& ({_v}W\otimes M)\otimes\vZ\\
 z\otimes w\otimes m &\mapsto& v^{-1}(z^{(-1)}).(w\otimes m) \otimes z^{(0)}=
 v(v^{-1}(z^{(-2)})).w\otimes v^{-1}(z^{(-1)}).m\otimes z^{(0)}
 \end{array}$$
 as well as the following equality of morphisms for all $W\in H\md\mod$:
 $$ \begin{array}{lllllll}
 {_v}Z\otimes (M\otimes W) &\stackrel{=}{\to}& ({_v}Z\otimes M)\otimes W & \to& 
 (M\otimes \vZ)\otimes W
 &\to& (M\otimes W)\otimes \vZ\\
 z\otimes m\otimes w &\mapsto& z\otimes m \otimes w  &\mapsto& 
  v^{-1}(z^{(-1)}).m\otimes z^{(0)}\otimes w &\mapsto& 
 v^{-1}(z^{(-2)}).m\otimes v^{-1}(z^{(-1)}).w \otimes z^{(0)}
 \end{array}$$
 $$ \begin{array}{lll}
 {_v}Z\otimes (M\otimes {_v}W) 
 &\to& (M\otimes {_v}W)\otimes\vZ\\
 z\otimes m\otimes w &\mapsto& v^{-1}(z^{(-1)}).(m\otimes w) \otimes z^{(0)}=
 v^{-1}(z^{(-2)}).m\otimes v^{-1}(z^{(-1)}).w \otimes z^{(0)}
 \end{array}$$
 \end{proof}
 
 \noindent
 We also discuss the connection to a different embedding\footnote{Thanks to the referee for asking this question}:
\newcommand{\M}{\mathcal{M}}
\renewcommand{\DH}{\mathrm{D}H}
\begin{remark}
    The authors of \cite{COZ97} define for a Hopf algebra $H$ the \emph{Quantum Brauer group} $\mathrm{BQ}(k,H)$, an analogue of the Brauer group. It consist of $H$-Azumaya $H$-Yetter-Drinfel'd algebras modulo $H$-Morita equivalence. In \cite{OZ98} they give a map $\pi:\Aut(H)\to \mathrm{BQ}(k,H)$ and determine the kernel. An elements in $A\in \mathrm{BQ}(k,H)$ gives rise to a $\DH\md\mod$-module category $A\md\mod$, i.e. an element in the Picard group. By \cite{DN12} in turn the Picard group maps to the Brauer-Picard group and hence to the group of braided autoequivalences - to be precise Thm. 4.3 states that the image of the Picard group consists precisely  of those braided autoequivalences which are trivializable on $H\md\mod\subset \DH\md\mod$. This is by construction exactly our subgroup $\bcat$ in the next section.\\
    
    We shall briefly sketch, how one can explicitly see the surjection of the subgroup $\Aut(H)$ to our subgroup $\vcat\subset\bcat$ through all these identifications: We first convince ourselves how the identity $v=\id\in \Aut(H)$ maps to the identity: The associated Azumaya algebra $A_{v^{-1}}$ is simply $\End{H}$ where $H$ is an $H$-Yetter-Drinfeld module with adjoint $H$-action and diagonal $H$-coaction. The module category $\mcat:=A_{v^{-1}}\md\mod$ has (as always) the single simple object $H$ with the above Yetter-Drinfeld structure. Now the implicit construction in \cite{DN12} Sec. 2.9 assigns to $\mcat$ the unique equivalence class of autoequivalences $\partial_\M\in\AutBr(\DH\md\mod)$, such that $\alpha^-\circ \partial_\M=\alpha^+$ are equal as module category morphisms, where $\alpha^\pm(X)$ means the module category morphisms given on objects by tensoring by $X\in \DH\md\mod$ and with module category morphism structure given by the braiding resp. the inverse braiding. Equal here means up to natural equivalence and indeed the double-braiding $X\otimes M\to M\otimes X\to X\otimes M$ turns out to be such a natural isomorphism between $X\otimes$ and itself that switches $\alpha^+,\alpha^-$. This shows how the Hopf-automorphism $\id$ indeed implies the braided autoequivalence $\partial=\id$ as expected.\\
    
    For arbitrary $v\in \Aut(H)$ the situation is more involved, but fairly similar: The Azumaya algebra is defined as $A_{v^{-1}}:=\End{H_{v^{-1}}}$ where $H_{v^{-1}}$ has again the diagonal coaction but a altered adjoint action $h.x=v^{-1}(h^{(2)})xS^{-1}(h^{(1)})$. This is not a Yetter-Drinfel'd module but fulfills the altered relation 
    $$(h.a)^{(0)}\otimes (h.a)^{(1)}= h^{(2)}.a^{(0)}\otimes v^{-1}(h^{(3)}) a^{(1)} S^{-1}(h^{(1)})$$
    Now if $\partial(X):={^{v^{-1}}}{_v}X$ is the Yetter-Drinfel'd module with modified action and coaction as in the theorem above, then one can roughly see that the double braiding maps
    $${^{v^{-1}}}{_v}X\otimes M\longrightarrow M \otimes {_v}X\longrightarrow M\otimes X$$
    so the double braiding in this sense gives an isomorphism $\alpha^-(\partial(X))\to \alpha^+(X)$ on objects, and as for identity the double braiding intertwines the braiding and negative braiding. \\
\end{remark}
 
\subsection{$\bcat$ induced from $H\md\mod$}~\\

Another rather obvious source of elements in $\BrPic$ is the induction functor 
from arbitrary monoidal equivalences; this of course contains the 
previous subgroup. While the bimodule category is given by 
definition, the image of the ENOM-functor requires some preparation:

Let $F:L\md\mod\to H\md\mod$ be a monoidal equivalence and let us 
consider the inverse $F^{-1}:H\md\mod\to L\md\mod$: We are assuming finite 
dimension, so $F^{-1}$ is given by $R\Box_{H^*}$ with $R={_f}H^*_\sigma$ an 
$L^*$-$H^*$-Bigalois object \cite{Sch91}, where $\sigma\in \Z^2(H^*,\CC)$ is a 
Hopf $2$-cocycle and $f:{_\sigma}(H^*)_{\sigma^{-1}}\to L^*$ is a Hopf algebra 
isomorphism from the Doi twist of $H^*$ to $L^*$. On objects $F^{-1}$ is just 
composing the coaction with $f$. E.g., for $H=\CC^G$ a dual groupring (but not 
always for a nonabelian groupring), due to the cocommutativity of $H^*=\CC[G]$ 
any Doi twist is equal to $H^*$ and $f$ is a choice of a group isomorphism 
$H^*\to L^*$. 

\begin{theorem}[\cite{MO98} Thm 2.7]
Given a $2$-cocycle $\sigma\in \Z^2(H^*,\CC)$, then we have the following 
category equivalence $\zcat(\mod(H^*))\to 
\zcat(\mod({_\sigma}(H^*)_{\sigma^{-1}}))$: Send $V$ to ${_\sigma}V$ with the 
same $H^*$-coaction and modified $H^*$-action 
$$f._\sigma 
v=\sigma(f^{(1)},v^{(-1)})\sigma^{-1}((f^{(2)}.v^{(0)})^{(-1)},f^{(3)})\cdot 
(f^{(2)}.v^{(0)})^{(0)}$$
and monoidal structure of the functor given by $\sigma$.
\end{theorem}

\noindent
We can now state:

\newcommand{\FZa}{{^{\sigma^{-1}\circ 
f\hspace{-.4cm}}}{_{f^{-1}}\hspace{0cm}}Z}

\begin{lemma}
  Let $F\in\EqMon(L\md\mod,H\md\mod)$ and $\sigma,f$ as above. The induction 
image of $F$ is by definition the bimodule category $\mcat:={_F}(H\md\mod)$. 

We 
claim that this element in $\BrPic(L\md\mod,H\md\mod)$ gives under the ENOM 
functor rise to the functor in $\EqBr(\D L\md\mod,\D H\md\mod)$ given on objects 
by $\Phi(\Ind(v)):Z\mapsto \FZa$ and with the monoidal structure of $F$. 

Here $\FZa$ means the $L$-module has been converted by $F$ to a $H$-module 
$F(Z)$ which means precompose the action by $f^{-1}$. On the other hand the  
$L^*$-action is pulled back to an ${_\sigma}(H^*)_{\sigma^{-1}}$-action 
by $f$ and further to a $H^*$-action by $\sigma^{-1}$ with the previous Lemma.

  In particular this defines a subgroup $\bcat\subset \BrPic(H\md\mod)$ which 
is the homomorphic image of the group $\AutMon(H\md\mod)$.
\end{lemma}
It is easy to see that the case $\sigma=1$ reduces to the elements (and the 
proof) in $\vcat$. 
\begin{proof}
  We denote the modified coaction by lower indices $z\mapsto z_{(-1)}\otimes 
z_{(0)}$. The relevant property of its definition is that $Z\mapsto \FZa$ is a 
braided category equivalence which coincides with $F$ on the level of modules. 
More formally $z_{(-1)}._F\;w=z^{(-1)}.w$. Using this property the proof works 
automatically as in the previous section: \\
  
  The  half-braiding (with modified coaction and action, but unmodified action 
on $M$!) 
  $${_F}Z\otimes M \rightarrow M\otimes \FZa$$
  $$z\otimes m\mapsto z_{(-1)}.m\otimes z_{(0)}$$ 
  gives clearly a natural isomorphism of $H$-modules, since we can write it as 
a braiding of $\FZa\otimes {_F}M'$ with $M'={_{F^{-1}}}M$.\\
 Then we check the coherence conditions using the relevant property:
  $$ \begin{array}{lllll}
 {_F}Z\otimes ({_F}W\otimes M) &\to& {_F}W\otimes ({_F}Z\otimes M)& \to& 
{_F}W\otimes (M\otimes \FZa)\\
 z\otimes w\otimes m &\mapsto& z^{(-1)}.w\otimes z^{(0)}\otimes m &\mapsto& 
  z^{(-1)}.w\otimes (z^{(0)})_{(-1)}.m\otimes (z^{(0)})_{(0)} 
 \end{array}$$
 $$ \begin{array}{lll}
 {_F}Z\otimes ({_F}W\otimes M) 
 &\to& ({_F}W\otimes M)\otimes\FZa\\
 z\otimes w\otimes m &\mapsto& z_{(-1)}.({_F}w\otimes m) \otimes z^{(0)}=
 (z_{(-1)})^{(1)}._F\;w\otimes (z_{(-1)})^{(2)}.m\otimes z^{(0)}
 \end{array}$$
 as well as the more trivial relation
 $$ \begin{array}{lllllll}
 {_F}Z\otimes (M\otimes W) &\stackrel{=}{\to}& ({_F}Z\otimes M)\otimes W & \to& 
 (M\otimes \FZa)\otimes W
 &\to& (M\otimes W)\otimes \FZa\\
 z\otimes m\otimes w &\mapsto& z\otimes m \otimes w  &\mapsto& 
  z_{(-1)}.m\otimes z_{(0)}\otimes w &\mapsto& 
 z_{(-2)}.m\otimes z_{(-1)}.w \otimes z_{(0)}
 \end{array}$$
 $$ \begin{array}{lll}
 {_F}Z\otimes (M\otimes {_F}W) 
 &\to& (M\otimes {_F}W)\otimes\vZ\\
 z\otimes m\otimes w &\mapsto& z_{(-1)}.(m\otimes w) \otimes z_{(0)}=
 (z_{(-1)})^{(1)}.m\otimes (z_{(-1)})^{(2)}.w \otimes z^{(0)}
 \end{array}$$
\end{proof}

\subsection{$\ecat$ induced from $H^*\md\mod$}\label{sec_subgroupE}~\\

Since $\zcat(H\md\mod)\cong \zcat(H^*\md\mod)$ we may as well induce up from 
$\AutMon(H^*\md\mod)$, which is in general not related to $\AutMon(H\md\mod)$ - 
except the common subgroup $\AutHopf(H)\cong \AutHopf(H^*)$. Here by definition 
$F\in \AutMon(H^*\md\mod)$ induces the $H^*\md\mod$-bimodule category 
${_F}(H^*\md\mod)$ and the image of $F$ under the ENOM functor in 
$\zcat(H\md\mod)\cong \zcat(H^*\md\mod)$ is dual to the last section. However, 
it is not clear what the $H\md\mod$-bimodule category associated to $F$ is; 
this is clarified by:

\newcommand{\FZb}{{_{\sigma^{-1}\circ f\hspace{-.25cm}}}{^{f^{-1}}}Z}
\begin{lemma}
  Let $F\in\EqMon(L^*\md\mod,H^*\md\mod)$ and consider again $F^{-1}$, which we 
write as cotensoring with a $L$-$H$-Bigalois object $R={_f}H_\sigma$ with 
$\sigma \in Z^2(H,\CC)$ and $f:{_\sigma}H_{\sigma^{-1}}\to L$. 
  We already know that (dually) the induction image of $F$ is by definition the 
$L^*$-$H^*$-bimodule category ${_F}(L^*\md\mod)$ and this gives under the ENOM 
functor rise to the functor in $\EqBr(\D L\md\mod,\D H\md\mod)$ given on 
objects by $\Phi(\Ind(F)):Z\mapsto \FZb$ and with the monoidal structure of 
$F$.\\
  We claim that this braided equivalence coincides with the image of the ENOM 
functor of the following invertible exact $L$-$H$-bimodule category: Let 
$\mcat=R\md\mod$ as $\CC$-linear category. The left and right coaction 
  $$ R\longrightarrow L\otimes  R
  \qquad  R\longrightarrow R\otimes H$$ 
  give by pull-back module category actions of $L\md\mod$ and $H\md\mod$ on 
$R\md\mod$.
  
  In particular this defines a subgroup $\ecat\subset \BrPic(H\md\mod)$ which 
is the homomorphic image of the group $\AutMon(H\md\mod)$.\footnote{This 
subgroup of $\BrPic$ has been considered first in a different approach of 
\cite{FMM14}; here we describe it as induction functor and give its image under 
the ENOM functor.}
\end{lemma}
\begin{proof}
  Let $M$ be an $R$-module. To prove our formula for $\Phi(\mcat)$ we need to 
guess a natural transformation:
  $$Z\otimes M\rightarrow M\otimes \FZb$$
  $$z\otimes m\mapsto \iota(z_{(-1)}).m\otimes z_{(0)}$$
  where we denote the $F$-modified coaction by lower indices $z_{(-1)}\otimes 
z_{(0)}\in H \otimes {^F}Z$ and the right-$H$-colinear cleaving identification 
map $\iota:H\cong H_\sigma$. To prove that this is indeed a natural 
transformation we need to check that it is an $R$-module map (it is clearly 
natural and bijective), so we act with some $\iota(H)\in R$ and wish to prove:
  \begin{align*}
   &\iota((\iota(h)^{(-1)}.z)_{(-1)}).\iota(h)^{(0)}.m\otimes 
(\iota(h)^{(-1)}.z)_{(0)}\\
   &\stackrel{?}{=} \iota(h)^{(0)}.\iota(z_{(-1)}).m\otimes 
\iota(h)^{(1)}._{\sigma^{-1}\circ f}z_{(0)}
  \end{align*}
  On the right hand side we use the right $H$-colinearity of $\iota$, on the 
  left hand side the left $L$-coaction on $R$ via $f$. Then we use that by 
  definition $\iota(a)\iota(b)=\sigma(a^{(1)},b^{(1)})\iota(a^{(2)}b^{(2)})$:
   \begin{align*}
  &\sigma((f(h^{(1)}).z)_{(-2)}, h^{(2)})\;
  \iota((f(h^{(1)}).z)_{(-1)}\cdot h^{(3)}).m\otimes (f(h^{(1)}).z)_{(0)}\\
   &\stackrel{?}{=} 
   \sigma(h^{(1)},z_{(-2)})\;
   \iota(h^{(2)}\cdot z_{(-1)}).m\otimes 
   h^{(3)}._{\sigma^{-1}\circ f}z_{(0)}
  \end{align*}
  To prove this relation is true the main issue is to 
  simplify the expression $(f(h^{(1)}).z)_{(-1)}$ using the 
  Yetter-Drinfeld-condition relation action and coaction, but since we have 
  lower-index (i.e. $F$-modified coaction) we need to also use the modified 
  action, which we obtain by adding and subtracting an appropriate cocycle. The 
  overall calculation is:
  \begin{align*}
   &\sigma((f(h^{(1)}).z)_{(-2)}, h^{(2)})\;
  \iota((f(h^{(1)}).z)_{(-1)}\cdot h^{(3)}).m\otimes (f(h^{(1)}).z)_{(0)}\\
  &=\underline{\sigma(h^{(1)},z_{(-2)})\sigma^{-1}(h^{(2)},z_{(-1)})}\;
  \sigma((f(h^{(1)}).z)_{(-2)}, h^{(2)})\;
  \iota((f(h^{(1)}).z)_{(-1)}\cdot h^{(3)}).m\otimes (f(h^{(1)}).z)_{(0)}\\
  &=\sigma(h^{(1)},z_{(-1)})\;
  \iota((h^{(2)}._{\sigma^{-1}\circ f} z_{(0)})_{(-1)}\cdot h^{(3)}).m
  \otimes (h^{(2)}._{\sigma^{-1}\circ f} z_{(0)})_{(0)} \\
  &=\sigma(h^{(1)},z_{(-2)})\;
  \iota(h^{(2)}z_{(-1)}\underline{S(h^{(4)})\cdot h^{(5)}}).m
  \otimes h^{(3)}._{\sigma^{-1}\circ f} z_{(0)} 
  \end{align*}
  
  Having established a natural transformation we check once again the coherence 
  conditions. We have equalities as follows for all $W\in L\md\mod$:
   $$ \begin{array}{lllllll}
  Z\otimes (W\otimes M) &\to& W\otimes (Z\otimes M)& \to& 
 W\otimes (M\otimes \FZb)
 &\stackrel{=}{\to}& (W\otimes M)\otimes\FZb\\
 z\otimes w\otimes m &\mapsto& z^{(-1)}.w\otimes z^{(0)}\otimes m &\mapsto& 
  z^{(-2)}.w\otimes \iota(z_{(-1)}).m\otimes z^{(0)} &&
 \end{array}$$
 $$ \begin{array}{lll}
 Z\otimes (W\otimes M) 
 &\to& (W\otimes M)\otimes\FZb\\
 z\otimes w\otimes m &\mapsto& \iota(z_{(-1)}).(w\otimes m) \otimes z^{(0)}=
 f(z_{(-2)}).w\otimes \iota(z_{(-1)}).m \otimes z^{(0)}
 \end{array}$$
 as well as for all $W\in H\md\mod$:
 $$ \begin{array}{lllllll}
 Z\otimes (M\otimes W) &\stackrel{=}{\to}& (Z\otimes M)\otimes W & \to&
(M\otimes \FZb)\otimes W
 &\to& (M\otimes W)\otimes \FZb\\
 z\otimes m\otimes w &\mapsto& z\otimes m \otimes w  &\mapsto& 
  \iota(z_{(-1)}).m\otimes z_{(0)}\otimes w &\mapsto& 
 \iota(z_{(-2)}).m\otimes z_{(-1)}.w \otimes z^{(0)}
 \end{array}$$
 $$ \begin{array}{lll}
 Z\otimes (M\otimes W) 
 &\to& (M\otimes W)\otimes\FZb\\
 z\otimes m\otimes w &\mapsto& \iota(z_{(-1)}).(m\otimes w) \otimes z_{(0)}=
 \iota(z_{(-1)})^{(0)}.m\otimes \iota(z_{(-1)})^{(1)}.w \otimes z^{(0)}
 \end{array}$$

\end{proof}

\subsection{$\pcat$ the partial dualizations}\label{sec_reflection}~\\

We now introduce an additional subset of elements in $\BrPic$ which are {\bf 
not} induced from monoidal equivalences, but constructed from the braided 
equivalence side of the ENOM functor. We will make thorough use of the second  
category equivalence $\Omega_X:\D X\md\mod\stackrel{\sim}{\to} \D X^*\md\mod$ 
\cite{BLS15} Thm. 3.20 for any Hopf algebra $X$ inside a \emph{braided base 
category} $\mathcal{X}$. The new $X^*$-action and -coaction on $\Omega(M)$ is 
as follows:

\begin{align*}
        \newcommand{\rhoomegaprime}{\raisebox{-.5\totalheight}{
        \begin{tikzpicture}
                \begin{scope}[scale=0.5]
                        \dAction{0}{2}{1}{-1}{\grau}{black}
                        \vLine{2}{1}{2}{3}{\grau}
                        \dPairing{1}{3}{1}{1}{\grau}{\tiny $\! \mathrm{eval}\!$}
                        \vLine{0}{2}{1}{3}{\grau}
                        \vLineO{1}{2}{0}{3}{black}
                        \vLine{2}{0}{1}{1}{black}
                        \vLineO{1}{0}{2}{1}{\grau}
                        \vLine{0}{3}{0}{4}{black}
                        \vLine{1}{-1}{1}{0}{\grau}
                        \vLine{2}{-1}{2}{0}{black}
                \end{scope}
                \draw (0.5 , -0.7) node {\tiny $X^*$};
                \draw (1 , -0.7) node {\tiny $M$};
                \draw (0 , 2.2) node {\tiny $M$};
        \end{tikzpicture}
        }
        }
        \newcommand{\deltaomegaprime}{\raisebox{-.5\totalheight}{
        \begin{tikzpicture}
                \begin{scope}[scale=0.5]
                        \dAction{1}{1}{1}{1}{\grau}{black}
                        \dCopairing{0}{1}{1}{1}{\grau}{\tiny $\!\! 
			\mathrm{coeval} \!\! \vspace*{-1mm}$}
                        \vLine{0}{1}{0}{5}{\grau}
                        \vLine{2}{0}{2}{1}{black}
                        \vLine{2}{2}{1}{5}{black}
                \end{scope}
                \draw (1 , -0.2) node {\tiny $M$};
                \draw (0 , 2.7) node {\tiny $X^*$};
                \draw (0.5 , 2.7) node {\tiny $M$};
        \end{tikzpicture}
        }
        }
        \newcommand{\omegatwoprime}{\raisebox{-.5\totalheight}{
        \begin{tikzpicture}
                \begin{scope}[scale=0.5]
                        \dAction{0}{1}{1}{-1}{\grau}{black}
                        \dAction{1}{3}{1}{1}{\grau}{black}
                        \dSkewantipode{1}{2}{\grau}
                        \vLine{2}{0}{2}{3}{black}
                        \vLine{0}{1}{1}{2}{\grau}
                        \vLineO{1}{1}{0}{2}{black}
                        \vLine{0}{2}{0}{4}{black}
                \end{scope}
                \draw (0.5 , -0.2) node {\tiny $M$};
                \draw (1 , -0.2) node {\tiny $N$};
                \draw (0 , 2.2) node {\tiny $M$};
                \draw (1 , 2.2) node {\tiny $N$};
        \end{tikzpicture}
        }
        }
        \Omega(M,\rho_M, \delta_M) = \left( M, \rhoomegaprime, \deltaomegaprime
\right) , \quad \Omega_2(M,N) = \omegatwoprime .
\end{align*}
with nontrivial monoidal structure $\Omega_2$ involving the inverse antipode. \\

\begin{lemma}
  The following $X\md\mod$-$X^*\md\mod$ bimodule category fulfills 
  the defining property of the preimage under the ENOM-functor of 
$\Omega$; it is {\bf not} necessarily invertible:\\
  As abelian category $\mcat=\mathcal{X}$ with trivial module 
category structure on either side (forgetting the $X,X^*$-module structures) 
but with nontrivial bimodule category structure
  $(V\otimes M) \otimes W\longrightarrow V\otimes (M\otimes W)$ given by 
  \begin{center}
            \begin{grform}
                        \begin{scope}[scale = .5]
			\draw[ultra thick, color = black, rounded corners] 
			 (0,0) -- (0,4);
                        \draw[ultra thick, color = black, rounded corners] 
			 (2,0) -- (2,4);
			 \draw[ultra thick, color = black, rounded corners] 
			 (4,0) -- (4,4);
			                                 
\dAction{-1}{2}{5}{-2}{\grau}{black}
\dAction{-1}{2}{1}{2}{\grau}{black}

                        \draw (0 , -0.6) node {$(V$};
                        \draw (2 , -0.6) node {$M)$};
                        \draw (4 , -0.6) node {$W$};
                        \draw (0 , 4.6) node {$V$};
                        \draw (2 , 4.6) node {$(M$};
                        \draw (4 , 4.6) node {$W)$};
                        \end{scope}
                \end{grform}
\end{center}
  where $V\in X\md\mod,\;W\in X^*\md\mod,\;M\in \mathcal{X}$.
\end{lemma}
\begin{proof}
 As natural equivalence $Z\otimes M\longrightarrow M\otimes \Omega(Z)$ we 
choose the braiding in the category $\mathcal{X}$, where $Z\in\D X\md\mod$ 
inside $\mathcal{X}$ and as objects in $\mathcal{X}$ we have  
$Z=\Omega(Z)$:
\begin{center}
            \begin{grform}
                        \begin{scope}[scale = 1]
                                \vLine{2}{0}{0}{2}{black}
                                \vLineO{0}{0}{2}{2}{black}
                                \end{scope}
                        \draw (0 , -0.3) node {$Z$};
                        \draw (0 , 2.3) node {$M$};
                        \draw (2 , -0.3) node {$M$};
                        \draw (2 , 2.3) node {$\Omega(Z)$};
                \end{grform}
\end{center}

 We check the coherence conditions that we have equalities of the following 
morphisms
\begin{center}
  \begin{minipage}{0.6\textwidth}
            \begin{grform}
                        \begin{scope}[scale = .5]
                        \vLine{2}{0}{0}{3}{black}
       			\vLineO{0}{0}{2}{3}{black}
                        \vLine{4}{2}{2}{6}{black}	
       			\vLineO{2}{3}{4}{6}{black}
                        \draw[ultra thick, color = black, rounded corners] 
			 (0,3) -- (0,6);
			 \draw[ultra thick, color = black, rounded corners] 
			 (4,0) -- (4,2);                         
			                                 
\dAction{-1}{0.5}{1}{-0.5}{\grau}{black}
\dAction{-1}{2.5}{1}{0.5}{\grau}{black}
 \draw[ultra thick, color = \grau, rounded corners] 
			 (-1,.5) -- (-1,2.5); 

                        \draw (0 , -0.6) node {$Z$};
                        \draw (2 , -0.6) node {$(W$};
                        \draw (4 , -0.6) node {$M)$};
                        \draw (0 , 6.6) node {$W$};
                        \draw (2 , 6.6) node {$(M$};
                        \draw (4 , 6.6) node {$\Omega(Z))$};
                        \end{scope}
                        \end{grform}
                        =$\qquad$
                        \begin{grform}
                         
                        \begin{scope}[scale = .5]
                        \vLine{2}{0}{0}{3}{black}
      			\vLineO{0}{0}{2}{3}{black}
                        \vLine{4}{2}{2}{6}{black}			        
			\vLineO{2}{3}{4}{6}{black}
                        \draw[ultra thick, color = black, rounded corners] 
			 (0,3) -- (0,8);
			\draw[ultra thick, color = black, rounded corners] 
			 (2,6) -- (2,8);
			 \draw[ultra thick, color = black, rounded corners] 
			 (4,6) -- (4,8);
			 \draw[ultra thick, color = black, rounded corners] 
			 (4,0) -- (4,2);  
			 
			\dAction{-1}{7}{5}{-1}{\grau}{black}
			\dAction{-1}{7}{1}{1}{\grau}{black}
			 
                        \draw (0 , -0.6) node {$Z$};
                        \draw (2 , -0.6) node {$(W$};
                        \draw (4 , -0.6) node {$M)$};
                        \draw (0 , 8.6) node {$W$};
                        \draw (2 , 8.6) node {$(M$};
                        \draw (4 , 8.6) node {$\Omega(Z))$};
                        \end{scope}
                \end{grform}
  \end{minipage}
  \end{center}
 as well as of the following morphisms involving the modified action on 
$\Omega(Z)$:
\begin{center}
  \begin{minipage}{0.6\textwidth}
            \begin{grform}
                        \begin{scope}[scale = .5]
                        \vLine{2}{0}{0}{3}{black}
       			\vLineO{0}{0}{2}{3}{black}
                        \vLine{4}{2}{2}{6}{black}	
       			\vLineO{2}{3}{4}{6}{black}
                        \draw[ultra thick, color = black, rounded corners] 
			 (0,3) -- (0,8);
			 \draw[ultra thick, color = black,rounded corners] 
			 (2,6) -- (2,8);
			 \draw[ultra thick, color = black, rounded corners] 
			 (4,6) -- (4,8);
			 \draw[ultra thick, color = black, rounded corners] 
			 (4,0) -- (4,2);                         
			                                 
\dAction{1}{3}{1}{0.5}{\grau}{black}
\dAction{1}{6}{1}{-.5}{\grau}{black}
 \draw[ultra thick, color = \grau, rounded corners] 
			 (.5,3) -- (.5,6); 
			 
			\draw (.8 , 6.3) node {\tiny eval};
                        \draw (1 , 2.7) node {\tiny coeval};

                        \draw (0 , -0.6) node {$(Z$};
                        \draw (2 , -0.6) node {$M)$};
                        \draw (4 , -0.6) node {$W$};
                        \draw (0 , 8.6) node {$(M$};
                        \draw (2 , 8.6) node {$W)$};
                        \draw (4 , 8.6) node {$\Omega(Z)$};
                        \end{scope}
                        \end{grform}
                        =$\qquad$
                        \begin{grform}
                         
                        \begin{scope}[scale = .5]
                        \vLine{2}{2}{0}{5}{black}
      			\vLineO{0}{2}{2}{5}{black}
                        \vLine{4}{4}{2}{8}{black}			        
			\vLineO{2}{5}{4}{8}{black}
                        \draw[ultra thick, color = black, rounded corners] 
			 (0,0) -- (0,1);
			 \draw[ultra thick, color =black, rounded corners] 
			 (0,5) -- (0,8);
			\draw[ultra thick, color = black, rounded corners] 
			 (2,0) -- (2,2);
			 \draw[ultra thick, color = black, rounded corners] 
			 (4,0) -- (4,2);
			 \draw[ultra thick, color = black, rounded corners] 
			 (4,0) -- (4,4);  
			 
			\dAction{-1}{1}{5}{-1}{\grau}{black}
			\dAction{-1}{1}{1}{1}{\grau}{black} 
			 
                        \draw (0 , -0.6) node {$(Z$};
                        \draw (2 , -0.6) node {$M)$};
                        \draw (4 , -0.6) node {$W$};
                        \draw (0 , 8.6) node {$(M$};
                        \draw (2 , 8.6) node {$W)$};
                        \draw (4 , 8.6) node {$\Omega(Z)$};
                        \end{scope}
                \end{grform}
  \end{minipage}
\end{center}
  \end{proof}

  Suppose now we have a projection $\pi:H\to A$ which means we can write $H=K 
\rtimes A$ where the coinvariants $K=H^{\coin{\pi}}$ is a Hopf algebra in the 
braided category $\D A\md\mod$. Then we can construct two Hopf algebras:
  $$r'(H):=K^*\rtimes A\qquad r(H):= \Omega_A(K)\rtimes A^*$$
  and category equivalences $\D H\md\mod \to \D\; r(H)\md\mod$ and $\D H 
\md\mod\to \D \; r'(H)\md\mod$.\\

Our previous lemma applied to $\mathcal{X}=\D A\md\mod$ gives a 
$\zcat(H\md\mod)$-$\zcat(r(H)\md\mod)$-bimodule category $\mcat'=\mathcal{X}$, 
which is in general not invertible. But there is an invertible sub-bimodule 
category stable under the structure maps, namely $A\md\mod$ ($M$ appears only 
as undercrossing). This shows for the first part:
\begin{corollary}
  The element $\D H\md\mod \to \D\; r'(H)\md\mod$  
  is the image under the ENOM functor of the module category 
  $\mcat=A\md\mod$
  with module structure given by the tensor product $\otimes_\CC$ in 
  $A\md\mod$, forgetting $K$- resp. $K^*$-module structure, and a nontrivial 
  bimodule category structure given by the previous lemma using the pairing 
  between $K,K^*$. 
\end{corollary}

\noindent
Similarly one constructs vice-versa:

\begin{corollary}
  The element $\D H\md\mod \to \D\; r(H)\md\mod$  
  is the image under the ENOM functor of the module category 
  $\mcat=(K\md\mod)_{\Omega_A^{-1}}$
  with module structure given by the tensor product $\otimes_\CC$ in 
  $K\md\mod$, for the right side after precomposing with $\Omega_A^{-1}$, 
  forgetting $A$- resp. $A^*$-module structure, and a nontrivial 
  bimodule category structure given by the previous lemma using the pairing 
  between $A,A^*$. 
\end{corollary}

\begin{example}
 The extremal case is a full dualization $r'$ with $A=1$ or equivalently $r$ 
with $K=1$. In this case we obtain the (in this case invertible) 
$H\md\mod$-$H*\md\mod$-bimodule category $\mcat=\Vect$ from the Lemma with 
bimodule category structure given by the pairing of $H$ and $H^*$.   
\end{example}

\noindent
  Very similar formulae construct dually $H^*\md\mod$-$r(H)^*\md\mod$-bimodule 
categories. \\
  
  Of particular interest are cases where $r'(H)\cong H$ resp. $r(H)\cong H$ 
which is the case for self-dual Yetter-Drinfeld Hopf algebra $K$ resp. 
self-dual Hopf algebra $A$ and $\Omega$-self-dual Yetter-Drinfeld module $K$. 
For these cases partial dualizations give rise to elements in 
$\BrPic(H\md\mod)$.   

\begin{remark}
 The bimodule categories should be equivalent to something like 
$\mcat:=(K\otimes K^*)_\lambda \rtimes A\md\mod$ resp.  $\mcat:= K\rtimes 
(A\otimes A^*)_\lambda\md\mod$ for the Bigalois object $(K\otimes K^*)_\lambda$ 
given by the evaluation pairing $K\otimes K^*\to \CC$ - and with trivial 
bimodule category structure.
\end{remark}
\begin{remark}
  Partial dualizations can be used to conjugate different forgetful functors 
$\zcat(\cat)\to \cat$ and hence many different induction functors from $\cat$. 
Our approach can be seen as the hope that this exhausts a large amount of 
different forgetful functors. 
\end{remark}
\begin{remark}
An important fact is that partial dualizations in our (narrow) definition 
depend on the precise Hopf algebra i.e. is not invariant under monoidal 
representation category equivalence. This can lead to the effect that 
$H\md\mod\cong H'\md\mod$ where $H$ has a semi direct decomposition while $H'$ 
has not, but still both centers carry the respective partial actualization. 
This can be either avoided by reformulating the above construction 
categorically (both categories have a semi direct-product-like decomposition) 
or by accepting, that partial dualizations can arise from any monoidally 
equivalent presentation. Compare the group example \ref{exm_exoticReflection} 
below.
\end{remark}

\section{Examples}\label{sec_examples}

\subsection{Groups}

We discuss all module categories and braided equivalences for the case 
$H=\CC^G$ with $G$ a finite group i.e. $H\md\mod=\Vect_G$. The module 
categories can be in this case be check against the explicit description:

\begin{lemma}[\cite{Dav10} Cor. 3.6.3 \cite{NR14} Prop. 
5.2]\label{lm_invertibleBimoduleCategories}
Invertible bimodule categories over $\Vect_G$ are in bijection with pairs 
$(B,\eta)$ where $U\subset G\times G^{op}$ a subgroup and $\eta\in 
H^2(B,\CC^\times)$ such that
\begin{itemize}
 \item $U(G\times 1)=U(1\times G^{op})=G\times G^{op}$
 \item $U_1=U\cap (G\times 1)$ and $U_2=U\cap (1\times G^{op})$ are abelian.
 \item $\eta(h_1,h_2)\eta^{-1}(h_2,h_1)$ is a nondegenerate pairing $U_1\times 
U_2\to \CC^\times$
\end{itemize}
In this case $\mcat=\Vect_{(G\times G^{op})/U}$ is the $\CC$-linear category of 
vector spaces graded by $U$-cosets \cite{O03}. The Lemma holds similarly for 
invertible $\Vect_G'$-$\Vect_G$-bimodule categories.
\end{lemma}

The braided equivalences of the center can be described very explicitly using 
the following well-known description:
\begin{lemma}
  $\zcat(\Vect_G)$ is semisimple and the simple objects are $\ocat_g^{\chi}$ 
where $[g]\subset G$ is a conjugacy class and $\chi$ an irreducible character 
of the centralizer $\Cent(g)$
\end{lemma}~\\

\subsubsection{We discuss the group $\vcat$}~ \\

Let $v:G'\to G$ be a group isomorphism. The corresponding invertible 
$\Vect_{G'}$-$\Vect_G$-bimodule category is given ${_v}(\Vect_G)$. This 
corresponds to the choice $G\cong U\subset G'\times G^{op}$ the graph of $v$ 
and $U_1=U_2=\{1\},\eta=1$. \\

\noindent
The ENOM functor assigns to this the following category equivalence of the 
centers:
$$\ocat_g^\chi\longmapsto \ocat_{v(g)}^{\chi(v^{-1}(\bullet))}$$~\\

\subsubsection{We discuss the group $\bcat$}~ \\

Let $F:\Vect_{G'}\to \Vect_G$ a monoidal equivalence: It is given on objects by 
a group isomorphism $v:G'\to G$ and the monoidal structure by a $2$-cocycle 
$\mu \in H^2(G',\CC^\times)$, which defines a Bigalois object $\CC_\sigma[G']$ 
with left coaction composed with $v$. Respective, the monoidal equivalence 
$F^{-1}$ is given by $f=v^{-1}$ and the $2$-cocycle 
$\sigma(g,h)=\mu^{-1}(v^{-1}(g),v^{-1}(h))$. The invertible 
$\Vect_{G'}$-$\Vect_G$-bimodule category is again given by definition by 
$\mcat={_F}(\Vect_G)$, which corresponds again to the choice $G\cong U\subset 
G'\times G^{op}$ the graph of $v$ and $U_1=U_2=\{1\}$ but now includes 
nontrivial $\eta$.\\

\noindent
The ENOM functor assigns to this the following category equivalence of the 
centers
$$\ocat_g^\chi\longmapsto 
\ocat_{v(g)}^{\chi(v^{-1}(\bullet))\frac{\mu(v^{-1}(\bullet),g)}{\mu(g,v^{-1}
(\bullet))}}$$
with nontrivial monoidal structure given by $\mu$ on the coaction.\\

\begin{remark}
 It is informative to also look at the bimodule categories from the dual  
perspective of the subgroup $\ecat$ of $\Rep(G')$-$\Rep(G)$-bimodule categories, 
where we obtain $\mcat={_v}(\CC_\sigma[G]\md\mod)$. 
\end{remark}~ \\

\subsubsection{We discuss the group $\ecat$}~\\

The monoidal equivalences $\Rep(G')\to H\md\mod$ are given by Bigalois objects 
${_f}R_{\Rep(G')}$ where $H$ is the Doi twist of $\CC[G']$ and $f\in 
\AutHopf(H)$. By \cite{Dav01} the Galois objects are given by pairs $(S,\eta)$ 
where $S$ is a subgroup of $G'$ and $\eta\in Z^2(S,\CC^\times)$ nondegenerate; 
then the Galois object is an induced representation $R={_f}(\CC^{G'}\otimes_{S} 
\CC_\eta[S])$. The Hopf algebra $H$, being the Doi twist of $\CC[G']$, is fixed 
up to isomorphism $f$ by the choice $S,\eta$. In particular obtaining again a 
group algebra $H=\CC[G]$ is equivalent to $S$ being normal abelian and the 
cohomology class $[\eta]$ being conjugation invariant. The isomorphism type of 
$G$ is a certain extension $\hat{S}\mapsto G\to G'/S$ determined by $\eta$.\\
In particular it is sufficient (but not necessary) to achieve $G'\cong G$ that 
$\eta$ is conjugation invariant as a $2$-cocycle. This additional condition is 
(see e.g. \cite{LP15}) equivalent to so-called {\it laziness}. In particular 
the extension $G$ is isomorphic to $G'$ by the trivial isomorphism (identity on 
$G'/S$ and the nondegenerate form defined by $\alpha$ identifying $S\cong 
\hat{S}$) and the additional morphism $f$ is actually a Hopf algebra 
isomorphism induced by a group isomorphism $v:G\to G'$. In this case we may 
assume $v=\id$ without loss of generality and realize $v\in\vcat$ as above.\\

The corresponding invertible $\Rep(G')$-$H\md\mod$-bimodule category induced by 
$F$ has been shown to be $R$-mod. To link this to the description in Lemma 
\ref{lm_invertibleBimoduleCategories} we observe that since $\CC_\eta[S]$ is by 
assumption a simple algebra, we have a category equivalence $R\md\mod \cong 
\CC^{G'/S}\md\mod=\Vect_{G'/S}$. In the lazy case it is easy to check that the 
following data in the Lemma describes our bimodule category: Identifying 
$G'/S=G/\hat{S}$, denoting the quotient map by $\pi$ and identifying 
$G^{op}\cong G$ via inverse we take 
$$U=\{(g',g)\in G'\times G\mid \pi(g')=v(\pi(g)^{-1})\}\qquad (S\times 
\hat{S})\to U\to G'/S$$
In particular $U_1=U\cap (G'\times 1)=S$ and $U_2=U\cap (1\times G)=\hat{S}$. 
There is a diagonal quotient $U\to S\times G'/S$, pulling back the $2$-cocycle 
$\eta$ gives a $2$-cocycle on $U$ which is nondegenerate on $S\times 
S,\;\hat{S}\times \hat{S}$ and $S\times\hat{S}$ as necessary.\\

The ENOM-functor assigns to this the category equivalence of the centers 
obtained above. It can be worked out for a given $\ocat_g^\chi$ by decomposing 
the induced representation according to the modified coaction, and the monoidal 
structure is given by that of $F$, but there is no convenient group-theoretic 
formula for this. We work out the following case:
\begin{lemma}
 The formula from Section \ref{sec_subgroupE} reduces for a lazy
 \footnote{This ``lazy'' here is much less critical  than in \cite{LP15}, where 
we classify lazy {\it braided} autoequivalences of the Drinfeld center. In the 
present approach it is merely a technical inconvenience that we have good 
explicit formulae only for (still non-lazy) induction from a lazy {\it 
monoidal} autoequivalence of $\Rep(G)$. Does the given group-theoretic formula 
continue to hold for nonlazy monoidal equivalences?}
  monoidal equivalence $\Rep(G)\to \Rep(G)$ given by $S,\eta,v=\id$ as follows 
on objects $\ocat_1^V$:\\
 Let the restriction of the irreducible $G$-representation $V$ to $S$ (abelian, 
normal) be decomposed according to Clifford theory into irreducible 
representations $V=\bigoplus_{i=1}^t E_i\otimes V_i$, where conjugation of $G$ 
acts transitively on the $1$-dimensional $S$-representations $V_i=\CC_{\chi_i}$ 
and all the multiplicity spaces $E$ (trivial $S$-representations) are of same 
dimension. Use the nondegeneracy of $\eta$ to identify $\hat{S}\cong S$ to get 
a $G$-conjugacy class $[s_i]$ by 
$\chi_i(r)=\frac{\eta(r,s_i)}{\eta(s_i,r)}=:\langle r,s_i\rangle$. Then the 
centralizer of any $s_i$ is the corresponding inertia subgroup $I_i\subset G$ 
fixing $\chi_i$ and hence acting on $E_i$. Then we claim  
 $$\ocat_1^V \longmapsto \ocat_{[s_i]}^{E_i\otimes V_i}$$
\end{lemma}
\begin{proof}
 Because the lazy case allows without restriction in generality to choose 
$v=\id$ we have $F=\id$ on objects. Thus as representations 
$\Phi(\Ind(F))\ocat_1^V=\ocat_1^V=V$ and as  
$\Phi(\Ind(F))\ocat_g^\chi={_{\sigma^{-1}}}\ocat_1^V$. So it remains to 
determine the $\sigma^{-1}$-twisted coaction, which is the 
$\sigma^{-1}-$twisted $\CC^G$-action. We need to reformulate also $G$-action as 
$\CC^G$-coaction via $v\mapsto \sum_g e_g\otimes g.v$. We decompose 
$V=\bigoplus_{i=1}^t E_i\otimes V_i$ as asserted and check the twisted action of 
the projector $e_{s_i}$ for $s_i$ defined as asserted on $v\in E_j\otimes V_j$:
 \begin{align*}
  e_{s_i}._\sigma v
  &=\sum_{g,h\in G} \sigma^{-1}(e_{s_i}^{(1)},e_g)\cdot 
(h.e_{s_i}^{(2)}.g.v)\cdot \sigma(e_h,e_{s_i}^{(3)})  
 \end{align*}
 We now use our formula in \cite{LP15a}:
 $$\sigma(e_a,e_b)=\frac{\delta_{a,b\in S}}{|S|^2}\sum_{t,t'\in S} 
\eta(t,t')\langle t,a\rangle\langle t',b\rangle$$
 and the fact that $v$ is in grade $1\in G$ to evaluate our expression. Then we 
exploit the fact that for a nondegenerate pairing on an abelian group holds 
 $\frac{1}{|S|}\sum_{s's''=s}\langle x,s\rangle\langle s,y\rangle=\delta_{x,y}$ 
and hence 
 $\frac{1}{|S|}\sum_{s's''=s}\langle x,s'\rangle\langle 
y,s''\rangle=\delta_{x,y}\langle x,s\rangle$ and that any $r\in S$ acts on $v$ 
by the $1$-dimensional character $\chi_j(r)=\langle r,s_j\rangle$:
 \begin{align*}
  &=\sum_{g,h\in S, s_i's_i''=s_i} \frac{1}{|S|^2}\sum_{t,t'\in S} 
\eta^{-1}(t,t')\langle t,s_i'\rangle\langle t',g\rangle
  \cdot (hg.v)\\
  &\cdot \frac{1}{|S|^2}\sum_{t'',t'''\in S} \eta(t'',t''')\langle 
t'',h\rangle\langle t''',s_i''\rangle\\
  &=\frac{1}{|S|^2}\sum_{r\in S} \sum_{t,t'\in S} \eta^{-1}(t,t')\langle 
t,s_i\rangle\langle t',r\rangle
  \cdot (r.v)\cdot \eta(t',t)\\  
  &=\frac{1}{|S|^2}\sum_{r\in S} \sum_{t,t'\in S} \langle t',t\rangle \langle 
t,s_i\rangle\langle t',r\rangle\chi_j(r)\cdot v\\
    &=\frac{1}{|S|}\sum_{r\in S}  \langle s_i,r\rangle \langle r,s_j\rangle 
\cdot v= \delta_{i,j}\cdot v
 \end{align*}
 This shows that $E_j\otimes V_j$ has now a coaction grade $s_j$ as asserted. 
\end{proof}

\begin{example}\label{exm_exoticReflection} We also wish to give an example of 
induction for a non-lazy autoequivalence. Consider $\Sp_{2n}(\F_2)$ acting on 
$S:=\ZZ_2^{2n}$ with invariant symplectic form $\langle\bullet,\bullet\rangle$. 
There is a unique nondegenerate cohomology class $[\eta]\in H^2(S,\CC^\times)$ 
associated to the symplectic form, which is hence invariant, however no 
representing $2$-cocycle is not invariant. It is known (\cite{Dav01} Exm. 7.6) 
that this relates the semidirect product $G'=S\rtimes \Sp_{2n}(\F_2)$ and the 
nontrivial extension $G=S.\Sp_{2n}(\F_2)$ via the (then non-lazy) Bigalois 
object associated to $S,\eta$.\\ 

Of particular interest is the case $n=1$ where both groups are isomorphic 
$G\cong G'=\SS_4$ but still $v$ interchanges the conjugacy classes $[(12)]$ and 
$[(1234)]$ (with both $6$ elements) and is hence no Hopf algebra isomorphism. 
The non-lazy monoidal autoequivalence $F$ of $\SS_4$ interchanges the two 
$3$-dimensional representations $\chi_3,\chi_3\otimes \sgn$and is visible as 
symmetry in the character table. The induction of this $F$ would yield a 
bimodule category $\mcat=R\md\mod$ which would be described by a $U\subset 
\SS_4\times\SS_4$ containing tuples such as $((12),(1234))$.\\
\end{example}

\subsubsection{We discuss the elements $\pcat$}~\\

We first observe that $\CC^{G}$ seems to have no interesting semidirect 
decompositions, because of contravariance this would imply a left-split 
sequence of groups. On the other hand assume $G=N\rtimes Q$, then 
$H^*=\CC[G]=\CC[N]\rtimes\CC[Q]$. Next we observe that partial dualization 
$r(\CC[G])$ can never return a group ring (except for a direct product, for 
which it coincides with $r'$), because the coaction of $A$ on $K$ is trivial, 
so to be self-dual the action would have to be trivial as well resulting in a 
direct product. \\

So we consider partial dualization $r'$ on $H^*=\CC[G]=\CC[N]\rtimes\CC[Q]$ 
where $N$ is an abelian group and a self-dual $Q$-module. We have already 
derived in \cite{LP15} a formula for the action of $r'$ as a braided 
equivalence of $\zcat(\Vect_G)$ on objects $\ocat_1^\chi$:

Similar to $\ecat$, let the restriction of the irreducible 
$G$-representation $V$ to $N$ (abelian, 
normal) be decomposed according to Clifford theory into irreducible 
representations $V=\bigoplus_{i=1}^t E_i\otimes V_i$, use the paring to 
map the $1$-dimensional representations $V_i\in N^*$ to a 
a $G$-conjugacy class $[s_i] \subset N$. Then the 
centralizer of any $s_i$ is the corresponding inertia subgroup $I_i\subset G$ 
fixing $V_i$ and hence acting on $E_i$. Then we claim  
 $$\ocat_1^V \longmapsto \ocat_{[s_i]}^{E_i}$$
(the only difference is no $V_i$ appears in the centralizer action) \\

Also the corresponding module category $\Vect_Q$ is described in striking 
similarity to $\ecat$ by the same subgroup 
$$U=\{(g',g)\in G\times G\mid \pi(g')=\pi(g)^{-1}\}
\qquad (N\times\hat{N})\to U\to Q$$
where $\pi:G\to Q=G/N$. But compared to $\ecat$ the $2$-cocycle is 
different: Consider again the diagonal quotient $N\to U\to N\times Q$ and 
consider the Masumoto spectral sequence
$$1\to H^1(N\times Q,\CC^\times)\to H^1(U,\CC^\times)\to H^1(N,\CC^\times)\to$$
$$\to H^2(N\times Q,\CC^\times)\stackrel{pullback}{\longrightarrow} 
H^2(U,\CC^\times)_N\stackrel{form}{\longrightarrow}
(N\times Q)\otimes N\to H^3(N\times Q,\CC^\times)\to H^3(U,\CC^\times)$$
where the subindex $N$ means cohomologically trivial if restricted to the 
kernel $N$. For $\ecat$ we took the pullback of a $2$-cocycle on $N$, now we 
should take the preimage of our nondegenerate form on $N\times N$, which 
becomes trivial in $H^3$ and is hence in the image.~\\ 

\subsubsection{Example: Elementary abelian groups}\label{sec_Fp_AutBr}~\\

\noindent
    For $G=\F_p^n$ a finite vector space we know directly
    $$\Aut_{br}(\DG\md\mod)=
    \begin{cases}
    O_{2n}(\F_p), & p\neq 2\\
    Sp_{2n}(\F_p), & p=2
    \end{cases}$$
    For abelian groups, all $2$-cocycles over $\DG$ are lazy and the
    results of \cite{BLS15} gives a product decomposition of
    $\BrPic(\Rep(G))$. The subgroups in question are
    \begin{itemize}
    \item $\vcat\cong \Out(G) = \GL_n(\F_p)$.
    \item $\bcat=\Out(G)\rtimes (\F_p^n\wedge \F_p^n)$
    the latter as an additive group.
    \item $\ecat=\Out(G)\rtimes (\F_p^n\wedge \F_p^n)$
    the latter as an additive group.
    \item The set $\pcat$ consists of $n+1$ equivalence classes of 
partial dualizations for each possible dimension $d$ of a direct
    factor $\F_p^d\cong C\subset G$. Especially the full dualization on $C=G$  
conjugates
    $\bcat$ and $\ecat$. In this case the proposed decomposition is actually a  
double coset decomposition, which is a variant of the
Bruhat decomposition of $O_{2n}(\F_p)$ of type $D_n$
    for $2\nmid p$ resp. $Sp_{2n}(\F_2)$ of type $C_n$.\\
    More precisely, our result reduces to the Bruhat decomposition of the Lie
groups $C_n$ resp. $D_n$ relative to the parabolic subsystem $A_{n-1}$. In
particular there are $n+1$ double cosets of the parabolic Weyl group $\SS_n$,
accounting for the $n+1$ non-isomorphic partial dualizations on subgroups 
$\ZZ_p^k$ for $k=0,...,n$.
    \end{itemize}~ \\

\subsubsection{Examples for nonabelian groups}~\\

\noindent
Let $G$ be a nonabelian simple group, then 
\begin{itemize}
    \item $\vcat\cong \Out(G)$.
    \item $\bcat=\Out(G)\rtimes H^2(G,\CC^\times)$
    (the latter as an additive group).
    \item $\ecat=\vcat$ as there are no nontrivial abelian 
    normal subgroups.	
    \item The set $\pcat$ is empty as there are no nontrivial semidirect 
    factors.	
\end{itemize}
Let $G=\SS_3$, then $\BrPic(G)=\ZZ_2$ (see already \cite{NR14}), more precisely:
\begin{itemize}
    \item $\vcat\cong \Out(G)=1$.
    \item $\bcat=\vcat\rtimes H^2(G,\CC^\times)=1$.
    \item $\ecat=\vcat=1$ since the only nontrivial abelian normal 
    subgroup is cyclic and has hence no nontrivial cocycles. 
    \item The set $\pcat$ contains a nontrivial reflection $r'$ on the normal 
    subgroup $\langle(123)\rangle$. As an element in $\AutBr(\D\SS_3\md\mod)$ 
    it permutes the objects as follows\footnote{$triv,sgn,ref$ the irreducible 
representations of $\SS_3$ and $\pm$ the two $1$-dimensional representations of 
the centralizer $\ZZ_2$ of $(12)$ and $1,\zeta,\zeta^2$ the three 
$1$-dimensional representations of the centralizer $\ZZ_3$ of $(123)$.
Whether $\zeta,\zeta'$ are the same roots of unity 
depends on the right choice of the pairing on $\ZZ_3$.}:
  $$\hspace{1cm} \ocat_1^{triv},\;\ocat_1^{sgn},\;\ocat_1^{ref},\;
  \ocat_{(12)}^{+},\;\ocat_{(12)}^{-},\; 
  \ocat_{(123)}^{1},\;\ocat_{(123)}^{\zeta},\;\ocat_{(123)}^{\zeta^2}$$
  $$\;\longrightarrow\;
  \ocat_1^{triv},\;\ocat_1^{sgn},\;\ocat_{(123)}^{1},\;
  \ocat_{(12)}^{+},\;\ocat_{(12)}^{-},\;
    \ocat_{1}^{ref},\;\ocat_{(123)}^{\zeta'},\;\ocat_{(123)}^{\zeta'^2}
  $$
  
  As an invertible bimodule category $\mcat$ is is the abelian category 
$\ZZ_2\md\mod$ with highly nontrivial bimodule category constraint. \\

We remark already at this point, that the associated group-theoretical  
$\ZZ_2$-extension of $\SS_3\md\mod$ is the fusion category 
$(\hat{\sl}_2)_4\md\mod$ which decomposes as an abelian category 
to $\SS_3\md\mod\oplus \SS_3/\ZZ_3\md\mod$ with $3+2$ simple objects.   
\end{itemize}

More examples are discussed in \cite{BLS15} Sec. 6.~\\

\subsection{Taft algebra}~\\

As a example which is not of group type, we now discuss the Taft algebra, for 
which the description of the Brauer Picard group can be checked against the 
list of bimodule categories in \cite{FMM14} (although there is unfortunately no 
description of the Brauer Picard group):
\begin{definition}[Taft algebra]
  Let $q$ be a primitive $\ell$-th root of unity 
prime) 
  and let $T_q$ be the Hopf algebra generated by $g,x$ with relations and 
coproduct as follows: 
  $$g^\ell=1\qquad x^\ell=0\qquad xg=qgx$$
  $$\Delta(g)=g\otimes g\qquad \Delta(x)=g\otimes x+x\otimes 1$$
  
  $T_q$ has dimension $\ell^2$ and decomposes into a Radford biproduct product 
$T_q=K \rtimes A=\CC[x]\rtimes \CC[\ZZ_\ell]$ where the $A$-action and 
-coaction on $K$ is given by $g.x=qx$ and $\delta(x)=g\otimes x$. It is a 
self-dual Hopf algebra via the linear forms $g^*:g,x\mapsto q,0$ and 
$x^*:g,x\mapsto 1,1$.\\
  
  The Taft algebra appears naturally as the Borel part of the small quantum 
groups $u_{q^{-1/2}}(\sl_2)^+$. The Drinfel'd double $\D T_q$ is generated by 
two isomorphic Taft algebras $g,x$ and $g^*,x^*$ with relations
  $$x^*g=q^{-1}gx^*\qquad xg^*=q^{-1}g^*x
  \qquad xx^*-qx^*x=\frac{g-g^*}{q-q^{-1}}$$
  It has the full quantum group as quotient by the central element $gg^*-1$.    
 
\end{definition}

\noindent
We recall some well-known properties of this Hopf algebra:
\begin{fact}
$$\AutHopf(T_q)\cong\CC^\times \qquad \OutHopf(H)\cong \CC^\times / \langle 
q\rangle$$ 
where $c\in \CC^\times$ acts by $g,x\mapsto g,cx$. This is because the 
skew-primitive $x$ is determined uniquely up to scalar and the grouplike $g$ is 
determined by $x$; on the other hand the asserted map is a Hopf algebra 
automorphism. Conjugation by $g$ gives the inner automorphism $c=q$, so 
$\OutHopf(H)\cong \CC^\times / \langle q\rangle$  
\end{fact}
  \begin{fact}
  All irreducible $T_q$-modules are of the form $\CC_\chi$ and all 
indecomposable modules are of the form $\CC_\chi[x]/x^d$ for 
  $\chi\in\hat{\ZZ}_p$ any character of the group ring and $0< d < \ell$ 
\end{fact}
\begin{proof}
  Let $V$ be a finite-dimensional $T_q$-module. Let $v$ be a $g.$-Eigenvector 
to some Eigenvalue $\chi(g)$ defining a character of $\ZZ_p$. The relation 
$gxg^{-1}=qx$ shows that $x^k.v$ is a $g.$-Eigenvector to the Eigenvalue 
$q^{k}\chi(g)$ and then at last $x^\ell.v=0$. Hence the only irreducible 
representations are $1$-dimensional $\CC_\chi$ and all indecomposables are 
$\CC_\chi[x]/x^d$ of dimension $0< d< \ell$. Conversely, each module can be 
realized as a quotient of the regular representation.
\end{proof}
\begin{fact}
  There is a braided subcategory of $\zcat(T_q\md\mod)=\D T_q\md\mod$ 
determined by $gg^*-1$ acting by zero, which is equivalent to the category of 
$u_{q^{-1/2}}(\sl_2)\md\mod$. We denote the irreducible highest weight module 
by $V(\lambda)$ for weight $\lambda\in \frac{1}{2}\NN_0$.
\end{fact}

We first discuss the group $\vcat\cong \OutHopf(T_q)\cong \CC^\times / \langle 
q\rangle$. The effect of this as a monoidal autoequivalence seems negligible 
because one easily finds a natural transformation to the trivial 
autoequivalence by rescaling $x^kv\mapsto c^k\cdot x^kv$. However, a {\it 
monoidal} natural transformation will return the trivial autoequivalence with a 
nontrivial monoidal structure. This can be easily seen for the tensor product of 
two 2-dimensional indecomposables, which decomposes into 1- and 3-dimensional 
indecomposables which are rescaled differently; the more general formula for 
$\ecat$ below shows the resulting 2-cocycle systematically for 
$(a,b)=(c^\ell,0)$.\\

\noindent
The induced bimodule categories are $\mcat^\vcat_c:={_c}(T_q\md\mod)$.\\

The ENOM functor maps this to the braided equivalence of the Drinfel'd center 
induced by $x,x^*,g,g^-*\mapsto cx,c^{-1}x^*,g,g^*$. Again, this is equivalent 
to a functor that is trivial on objects but with nontrivial monoidal structure. 
\\ 

To determine the group $\ecat$ we need to know the Bigalois objects. This has 
been done in \cite{Sch00} and can today be understood in the context of 
nontrivial lifting \cite{M01}:
\begin{lemma}
The right Galois objects are as follows for any choice $a\in\CC^\times 
,b\in\CC$\footnote{The right Galois objects are isomorphic for all values 
$b\neq 0$, but not as Bigalois objects. There are differently scaled left 
coactions, but the latter can be rescaled to $1$ by a Bigalois isomorphism at 
the cost of $a$.}:
$$R_{a,b}=\langle \tilde{g},\tilde{x} 
\rangle/(\tilde{g}^\ell=a1,\;\tilde{x}^\ell=b1,\;\tilde{x}\tilde{g}=q\tilde{g}
\tilde{x})$$
$$\delta(\tilde{g})=\tilde{g}\otimes g\qquad \delta(\tilde{x})=1\otimes 
x+\tilde{x}\otimes g$$
These all become $T_q$-$T_q$-Bigalois objects $R_{a,b}$ with the left coaction:
$$\delta(\tilde{g})=g\otimes \tilde{g}\qquad \delta(\tilde{x})=1\otimes 
\tilde{x}+ x \otimes \tilde{g}$$
So $\ecat\cong\Bigal(T_q)\cong \CC^\times \ltimes \CC$ and the embedding of 
$\vcat\cong \CC^\times / \langle q\rangle$ goes via $c\mapsto (c^\ell,0)$.
\end{lemma}
The induced bimodule categories are $\mcat^\ecat_{a,b}:=R_{a,b}\md\mod$. These 
are the $\mathcal{L}$ in \cite{FMM14}. As a $\CC$-linear category this is 
$T_q\md\mod$ (for $b=0$) as discussed in $\vcat$ or $\Vect$ ($b\neq 0$), since 
in the latter case there is a unique simple module $M$ of dimension $\ell$.

The elements $\pcat$ for $T_q$ are particularly interesting and will be 
generalized later:
\begin{itemize}
\item Since the Taft algebra is self-dual, we have the full dualization 
$\star\in \AutBr(\zcat(T_q\md\mod))$ (i.e. $r$ for $K=1$ or equivalently $r'$ 
for $A=1$). It decomposes into $rr'$ below.
\item For the decomposition $T_q=\CC[x]\rtimes \CC[\ZZ_\ell]$ we have $K\cong 
K^*$ as Yetter-Drinfel'd Hopf algebra, so we have a partial dualization 
$r'\in\AutBr(\zcat(T_q\md\mod))$.
It acts on quantum group modules $V(\lambda)$ like a reflection.
\item For the decomposition $T_q=\CC[x]\rtimes \CC[\ZZ_\ell]$ we also have 
$A\cong A^*$ and $K\cong \Omega(K)$, so we also have a partial dualization 
$r\in \AutBr(\zcat(T_q\md\mod))$. 
\end{itemize}

\section{Applications}\label{sec_applications}

\subsection{Quantum groups and Nichols algebras}

We now discuss some applications of the previously defined general elements if 
applied to quantum groups.~ \\

\subsubsection{$\AutBr$ of nonabelian groups and Nichols algebras}~\\

We begin with a little demonstration of the effect of our subgroups of 
$\BrPic(\Rep(G))$ as subgroups $\AutBr(H)$ after the ENOM functor. Namely, the 
Nichols algebra $\Nic(M)$ associated to some $M\in \D H\md\mod$ is a 
fundamental construction with a universal property. It returns e.g. the Borel 
part of the quantum group $U_q(\g)^+$ over $H=\CC[\ZZ_\ell^{\mathrm{rank}}]$.\\

Thus it is as a vector space invariant under $\AutBr(\D H\md\mod)$. We want to 
argue that this completely explains certain coincidences in dimension that 
appeared during the classification of finite-dimensional Nichols algebras over 
nonabelian groups $G$. Some of these cases have been known\footnote{SL is 
indebted to E. Meir for explaining this to him.} and the only purpose of this 
section is to collect and unify the argument using our explicit results on 
$\BrPic(G)$ in the previous section. 

Needless to say, this game of changing the realizing group 
does not reveal much 
information about the Nichols algebra. 

\begin{itemize}
 \item The following was the first explained coincidence in terms of Doi twist 
in \cite{Ven12}:\\
Let $G=\SS_4$ and 
consider $V=\ocat_{(12)}^{-+}$ and $V'=\ocat_{(12)}^{--}$ where $\pm\pm$ 
indicates the $1$-dimensional character of the centralizer 
$\langle(12),(34)\rangle$. Our results show that these two objects are 
interchanged by the braided autoequivalence in $\bcat$ induced from the 
nontrivial $2$-cocycle of $\SS_4$ (which restricts to a nontrivial class on the 
centralizer).  Both Nichols algebras have dimension $24^2$.\\
 More generally let $G=\SS_n$ and consider  $V=\ocat_{(12)}^{-+}$ and 
$V'=\ocat_{(12)}^{--}$ where the centralizer is $\ZZ_2\times \SS_{n-2}$. Then 
again these two objects are interchanged by the 	unique $2$-cocycle 
inducing up to $\bcat$. In case $n=5$ both Nichols algebras are of finite 
dimension $4^45^26^4$. 
 \item Let $G=\SS_4$ and consider $V=\ocat_{(12)}^{--}$ and 
$V'=\ocat_{(1234)}^{-1}$. We claim that our results show that these two elements 
are interchanged by the braided equivalence in $\ecat$ induced by the nonlazy 
monoidal autoequivalence $F$ of $\Rep(\SS_4)$ defined by $S$ the Klein-4-group 
and its unique nondegenerate 2-cocycle. Namely, as objects in $\Rep(\SS_4)$ 
these are $\sgn+\chi_2+\chi_3\cdot\sgn$ respectively $\sgn+\chi_2+\chi_3$ 
(where $1+\chi_2,\;1+\chi_3$ are the permutation characters) and $F$ 
interchanges $[(12)],[(1234)]$ and $\chi_3,\sgn\cdot\chi_3$. Again these 
Nichols algebras have dimension $24^2$. 
 \item On the other hand $V=\ocat_{(12)}^{-+}$ and $V'=\ocat_{(1234)}^{-1}$ are 
directly related by the partial dualization on $S$, which is due to a relation 
in $\BrPic$.
 \item Let $G=\ZZ_5\rtimes \ZZ_5^\times$ and consider 
$V=\ocat_{i\rtimes 2}^{-1}$ and $V=\ocat_{i\rtimes 3}^{-1}$. One can easily see 
that these two objects are interchanged by an outer automorphism. The respective 
Nichols algebras have dimension $1280$. A similar connection holds between two 
Nichols algebras over $\ZZ_7\rtimes \ZZ_7^\times$.
\item Over the dihedral group $\DD_4=\langle x,y\mid x^2=y^2=(xy)^4=1\rangle$ 
with $8$ element there are four Nichols algebras of dimension $64$, that are 
all interchanged by the Brauer Picard group, which is $\SS_4$ by \cite{NR14}. 
\end{itemize}~ \\

\subsubsection{Braided autoequivalences of quantum groups}~\\

Already the well-known fact that $\BrPic(H\md\mod)\cong 
\AutBr(\D H\md\mod)$ has interesting implications for 
$H=\Nic(M)\rtimes\CC[G]$ as we have already seen in the Taft algebra case:\\

Quasi-triangular quantum groups $u_q(\g)$ can be obtained\footnote{In 
case $q$ has even order, care has to be taken at this point} as quotients of 
$\D H$ for suitable Nichols algebras by grouplikes. So the category $\D 
H\md\mod$ has the category $u_q(\g)\md\mod$ as subcategory. For a given element 
in $\BrPic(\Nic(M)\rtimes\CC[G])$ we can ask whether the braided 
autoequivalence associated by the ENOM functor fixes this subcategory, so we 
obtain a braided autoequivalence of $u_q(\g)\md\mod$.

This question seems quite easy to answer (and usually to answer positively) 
because it involves only knowledge about the action of the grouplikes $\CC[G]$ 
resp. $\CC[G\times G]$ in the double: More precisely, a sufficient condition is 
that the braided autoequivalence preserves the forgetful functor to $\D 
G\md\mod$, as is e.g. the case for the interesting elements in $\ecat$ we 
discuss below. More general criteria could be given. \\

\subsubsection{Induction images $\bcat,\ecat$}~\\

Since the notion of a Nichols algebra is self-dual, it suffices to restrict to 
study to $\ecat$ (compare to the Taft algebra), so we wish to know the Bigalois 
objects. This is in general difficult, but there has been significant 
progress in the context of liftings of Nichols algebras, which we want to 
briefly comment on:\\

A long-standing question is to classify algebras $L$ 
with $\gr(L)=\B(M)\rtimes \CC[G]$ and the conjecture stands, that all of these 
algebras are related by a $2$-cocycle Doi twist i.e. there exists a 
$L$-$H$-Bigalois object and hence there is a monoidal equivalence  
$F:\B(M)\rtimes \CC[G]\md\comod\to  L\md\comod$, see \cite{M01}\cite{AAIMV13}. 
In the former paper this has been observed for quantum groups, where the 
classification of pointed Hopf algebras produces families with free 
lifting parameters, which turn out to however all be related by $2$-cocycle 
deformations. In the latter paper an impressive program has been presented to 
systematically determine all different liftings for a given Nichols algebra.\\

Thus: For a given Nichols algebra, e.g. $u^+_q(\g)$, the $\BrPic$-groupoid 
contains large (multi-parameter) families of objects $L$ with different 
liftings, e.g. with deformed relations like $E_{\alpha_i}^{N_i}=\mu_i\in \CC$, 
all of which are connected by elements in $\ecat$. Note that this gives 
bimodule categories between categories $H\md\mod$ and $L\md\mod$ that are very 
different as categories.\\

\begin{remark}
  From a physical perspective it very interesting to study such defects 
  between different phases labeled $H\md\mod$ and $L\md\mod$, in particular  
  where $H$ is the Borel part of a quantum group and $L$ is a different 
  lifting. Take for example the relation $E_{\alpha_i}^{N_i}=\mu_i$, which  
  resembles closely what one has in finite $W$-algebras. All different liftings 
  of this type come from different subcategories (sectors) 
  of the Kac-Procesi-DeConcini-Quantum group where $E_{\alpha_i}^{N_i}$ is a 
  central element. The subcategories are enumerated by collections of $\mu_i$ 
  that are in bijection to points of the complex Lie group associated to $\g$. 
  In this view, all these bimodule categories (defects) between different 
  categories can actually be collected to bimodule categories between this new 
  large category.        
\end{remark}

Needless to say, these are not the only objects in $\BrPic$, at least not for 
general Nichols algebras, as the reflections $\pcat$ in the next two sections 
show. \\

\subsubsection{Partial dualization on the Cartan part}~\\

We want to now more thoroughly treat partial dualization on the Cartan part of a 
quantum Borel part of a quantum group $U_q(\g)$ and find relations to the 
$L$-dual of the respective Lie group, at least in the simply-laced case. 
We assume that the TFT side of our construction is actually related to 
$T$-duality; this could explain why an $L$-dual appears, see \cite{DE14}:\\

Let $H=U_q(\g)^\geq=U_q(\g)^+\rtimes \CC[\Lambda]$ where $\Lambda=\ZZ^{\rank}$ 
is a lattice (resp. a quotient at roots of unity) sitting between root- and 
weight-lattice of $\g$ i.e. $\Lambda_W\supset \Lambda\supset \Lambda_R$. The 
embedding and the scalar product determine the Yetter-Drinfel'd structure of 
$U_q(\g)^+$ via
$$K_\lambda.E_\alpha=q^{(\lambda,\alpha)} E_\alpha
\qquad \delta(E_\alpha)\mapsto K_\alpha\otimes E_\alpha$$
The choice of $\Lambda$ is parametrized by a subgroup of 
$\Lambda_W/\Lambda_R=\pi_1$ which determines the fundamental group of the 
respective complex Lie group, which parametrizes different topological 
coverings. Correspondingly the usual choice $\Lambda=\Lambda_R$ is the adjoint 
form and $\Lambda=\Lambda_W$ is often called the simply-connected form. \\

\begin{lemma}
 For $\g$ simply laced partial dualization $r$ on the Cartan part $\CC[\Lambda]$ 
interchanges $U_q(\g)^+\rtimes \CC[\Lambda]$ and $U_q(\g)^+\rtimes 
\CC[\Lambda^\vee]$; e.g. it interchanges adjoint and simply-connected form. In 
particular for small quantum groups at an $\ell$-th root of unity it 
interchanges  $u_q(\g)^+\rtimes \CC[\Lambda/\ell\Lambda^\vee]$ and 
$u_q(\g)^+\rtimes \CC[\Lambda^\vee/\ell\Lambda]$
\end{lemma}
\begin{proof}
 In the case of the Taft algebra this has been checked explicitly in our paper 
\cite{BLS15}. Take the obvious group pairing $\Lambda\times 
\Lambda^\vee\to\CC^\times$ given by $\lambda\otimes \mu\mapsto 
q^{(\lambda,\mu)}$. It gives in particular rise to a nondegenerate group 
pairing:  
 $$\Lambda/\ell\Lambda^\vee\times \Lambda^\vee/\ell\Lambda\to \CC^\times$$
 We need to convince ourselves that this dualization interchanges action and 
coaction. But this is clearly true
 $$K_\lambda._{r}E_\alpha :=\langle K_\lambda,K_\alpha \rangle E_\alpha= 
q^{(\lambda,\alpha)}\;E_\alpha$$
\end{proof}

\begin{corollary}
 Partial dualization as discussed above gives rise to the braided category 
equivalence between the Drinfel'd centers of $u_q(\g)^+\rtimes 
\CC[\Lambda/\ell\Lambda^\vee]$ and $u_q(\g)^+\rtimes 
\CC[\Lambda^\vee/\ell\Lambda]$. It restricts to a braided category equivalence 
between the respective quantum groups $u_q(\g)$ associated to $\Lambda$ and 
$\Lambda^\vee$.
\end{corollary}
\begin{corollary}
 Partial dualization as discussed above is the image under the ENOM functor of 
the module category $\mcat_r:=u_q(\g)^+\md\mod$, which is as $\CC$-linear 
category the Nichols algebra representation category and has a nontrivial 
bimodule category structure defined by $q^{(\lambda,\mu)}$ for 
$\lambda\in\Lambda$ and $\mu\in\Lambda^\vee$
\end{corollary}~\\

\subsubsection{Partial dualization and Weyl reflection}~\\

We now turn our attention to reflections of the Nichols algebra in the original 
sense: Let $M=\bigoplus_i M_i$ a decomposition of the object $M$ into simple 
objects, then $\alpha_i$ is a simple root for the Nichols algebra $\B(M)$ 
in the sense of \cite{AHS10}. For example for the 
semisimple complex finite-dimensional Lie algebra $\g$ we have 
$U_q(\g)^+=\Nic(M)$, resp. $U_q(\g)^+=\Nic(M)$ for roots of unity, for a 
choice of a $\ZZ^{\rank}$-Yetter-Drinfeld module $M=\bigoplus_i 
E_{\alpha_i}\CC$ where $\alpha_i$ a simple root in the usual sense.\\ 

Then the reflection of this Nichols algebra is the special 
case of a partial dualization $r$ with respect to the projection, see 
\cite{HS13}\cite{BLS15}
$$\pi_i:\;\Nic(M)\to \Nic(M_i)\qquad \Nic(M)=\Nic(M)^{\coin\pi}\rtimes 
\Nic(M_i)$$
For semisimple Lie algebras there is an algebra isomorphism $r(\Nic(M))\cong 
\Nic(M)$, namely Lusztig's reflection automorphism $T_{w_i}$ for the simple 
reflection $w_i$, but for general Nichols algebras these two algebras can be 
non-isomorphic. Nevertheless our results {\it (in cit. loc.)} show in all cases 
 a category equivalence
$$\zcat(\Nic(M)\md\mod)\cong \zcat(r(\Nic(M))\md\mod)$$
In particular for the Lie algebra case this restricts to a braided equivalence 
$T_{w_i}:U_q(\g)\md\mod\to U_q(\g)\md\mod$ and more general for every Weyl 
group element $w\in W$.\\

We now discuss the $\Nic(M)\md\mod$-$r(\Nic(M))\md\mod$-bimodule 
categories associated to these partial dualizations. This is interesting 
already in the Lie algebra case: Our results in Section \ref{sec_reflection} 
show that the preimage of there is a bimodule category
$$\mcat_{w_i}:=\Nic(M)^{\coin\pi}\md\mod$$
With left resp. right categorical action by $\Nic(M)\md\mod$ resp. 
$r(\Nic(M))\md\mod$, forgetting $\Nic(M)$ resp. $\Nic(M^*)$-action, and a 
nontrivial bimodule category constraint $(V\otimes M)\otimes W\cong V\otimes 
(M\otimes W)$ given by the evaluation map $\Nic(M)\otimes \Nic(M^*)\to \CC$.

\begin{remark} Iterating this procedure yields for every Weyl group element 
$w\in W$ a bimodule category 
$$\mcat_w:= U^+[w]\md\mod$$
It is worth mentioning that 
these are precisely the homogeneous coideal subalgebras of $U^+(\g)$; so it 
would be interesting to consider (and recognize in our ansatz) bimodule 
categories for all coideal subalgebras $C$, which are classified by \cite{HK11} 
to be character shifts $C=(\id\otimes \chi)\Delta U^+[w]$.  
\end{remark}~ \\

\subsection{Defects in 3D topological field theories}~\\

\newcommand{\Bord}{\mathrm{Bord}}
\noindent
An oriented $(3,2,1)$-extended TQFT is a symmetric monoidal weak $2$-functor: 
$$Z: \Bord^{or}_{3,2,1} \to 2\Vect $$                  
where $\Bord^{or}_{3,2,1}$ is the symmetric monoidal bicategory of oriented 
$3$-cobordisms and $2\Vect$ the symmetric monoidal bicategory of 
Kapranov-Voevodsky 2-vector spaces, thus objects of $2\Vect$ are $k$-linear, 
abelian, semisimple categories, morphisms are $k$-linear functors and 
$2$-morphisms are natural transformations. (See \cite{KV94}, \cite{Mo11} and the 
Appendix of \cite{BDSV15} for more details on $2\Vect$ and other targets). \\ 
Oriented $(3,2,1)$-extended TQFTs are classified by anomaly free modular tensor 
categories (by Thm. 2 in \cite{BDSV15}), where a functor $Z$ corresponds to the 
anomaly free modular tensor category $Z(S^1)$, which we also refer to as the 
category of bulk Wilson lines. For general modular tensor categories, such 
theories are called Reshetikhin-Turaev type theories. In the case the modular 
tensor category is $Z(S^1)=\zcat(\cat)$, the Drinfeld center of some fusion 
category $\cat$, such theories are called Turaev-Viro type theories. One can use 
the Reshetikhin-Turaev construction \cite{RT91}, which is essentially based 
on surgery on $3$-manifolds along links, to define a Reshetikhin-Turaev type 
theory explicitly.

A special case are {\it Dijkgraaf-Witten theories} with $Z(S^1) = 
\zcat(\Vect^{\omega}_G)$ where $\Vect^{\omega}_G$ is the category of $G$-graded 
vector spaces for some finite group $G$ and non-trivial associativity 
constraints determined by $3$-cocycles $\omega \in \Z^3(G,k^\times)$. If $\omega 
= 1$ the Dijkgraaf-Witten theory is called untwisted. Dijkgraaf-Witten 
theories can be realized explicitly by linearizing the category of 
principal $G$-bundles on a manifolds i.e. 
$$Z(\Sigma):=\mathrm{Fun}(\mathrm{Bun}_G(\Sigma),\Vect)
\qquad \mathrm{Bun}_G(\Sigma)\cong \mathrm{Hom}(\pi_1(\Sigma),G)$$
and $\Z(M)$ by so-called pull-push-construction, that sums over all possible 
continuations of bundles on $\Sigma$ to $M$, see e.g. \cite{FPSV14}.\\

We now consider additional data on the manifold, namely {\it surface defects}: 
These are codimension $1$ submanifolds. Suppose for example 
$\Sigma_\text{Transm}=S^1\times [-1,1]$ and a middle circle 
belonging to a defect $d$, then the two bounding circles get assigned some 
$Z(S^1\times \{-1\})=\zcat(\cat)$ and $Z(S^1\times \{1\})=\zcat(\dcat)$ and the 
defect a bimodule category $Z(S^1\times \{0\})={_\cat}\mcat_\dcat$. On 
the other hand the TFT assigns to this situation a (due to the defect possibly 
nontrivial) morphism $\zcat(\cat)\to\zcat(\dcat)$:
\begin{center}
\includegraphics[scale=0.4]{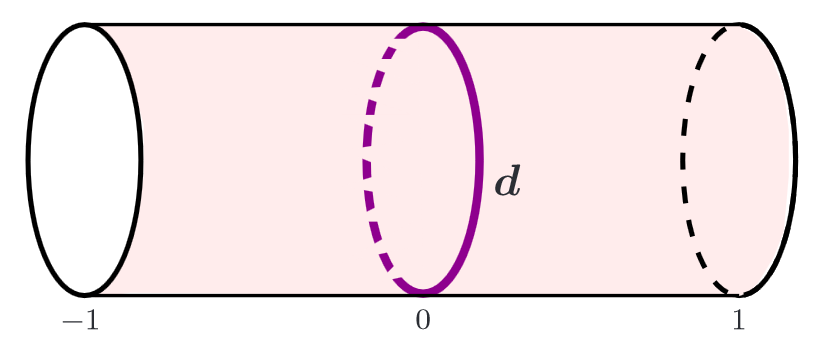}
\end{center}

This becomes a monoidal functor with the monoidal structure given by 
$Z(M_{pants\;defects})$ for the following $3$-manifold with defect: (the 
cylinder has been flattened to a annulus) 

\begin{center}
\includegraphics[scale=0.4]{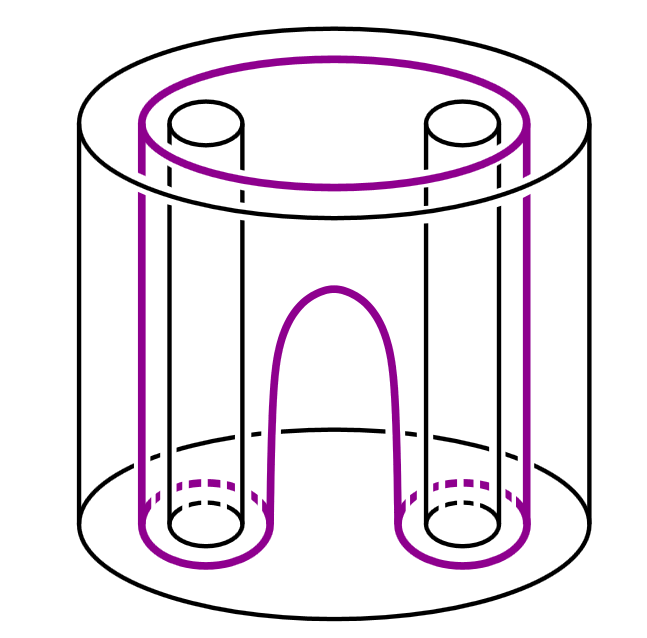}
\end{center}

The coherence condition is checked by noticing that the following two 
manifolds are diffeomorphic:

\begin{center}
\includegraphics[scale=0.3]{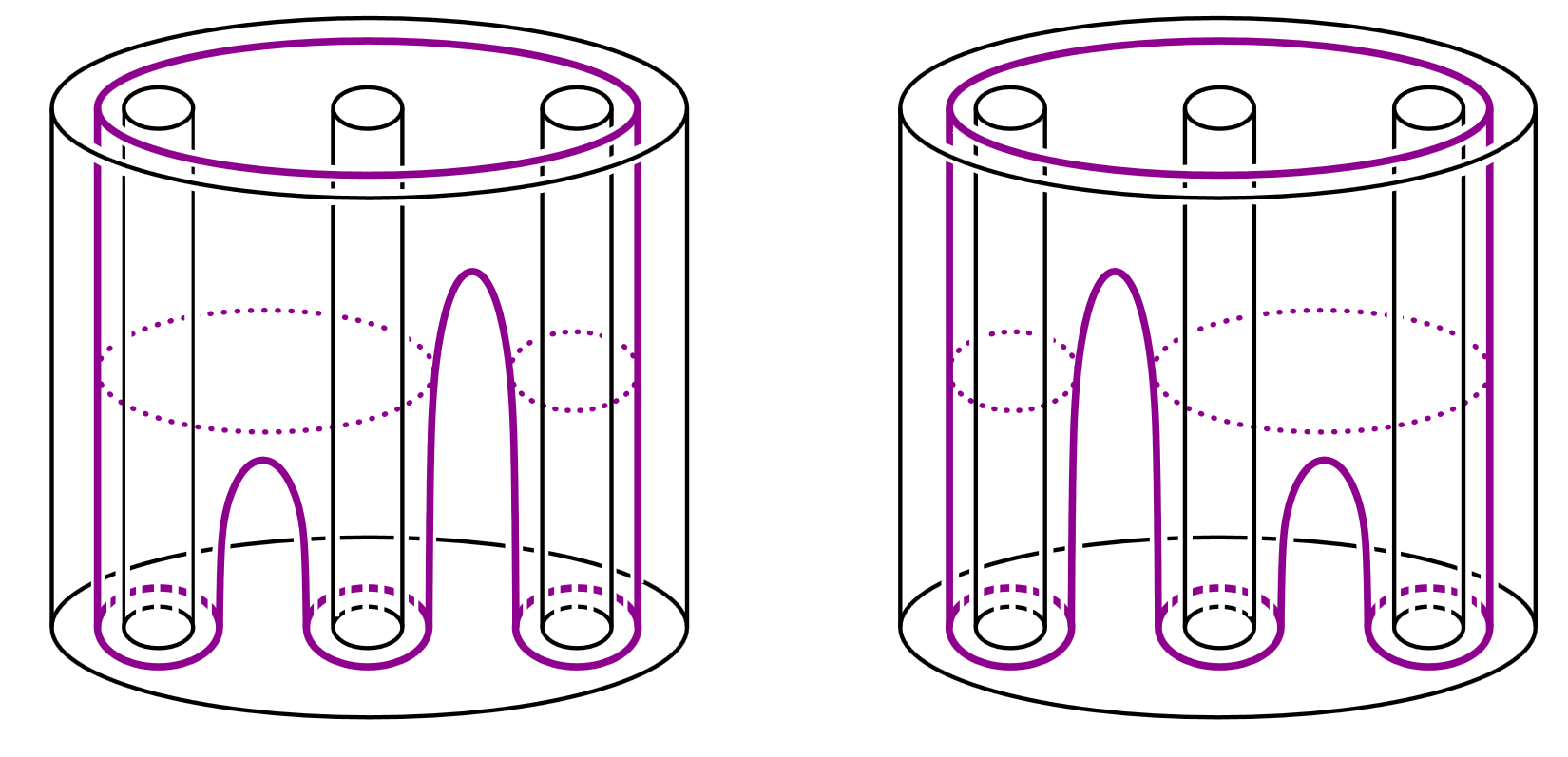}
\end{center}

and the following two diffeomorphic manifolds show the functor 
$Z(\Sigma_\text{Transm})$ is braided:

\begin{center}
\includegraphics[scale=0.3]{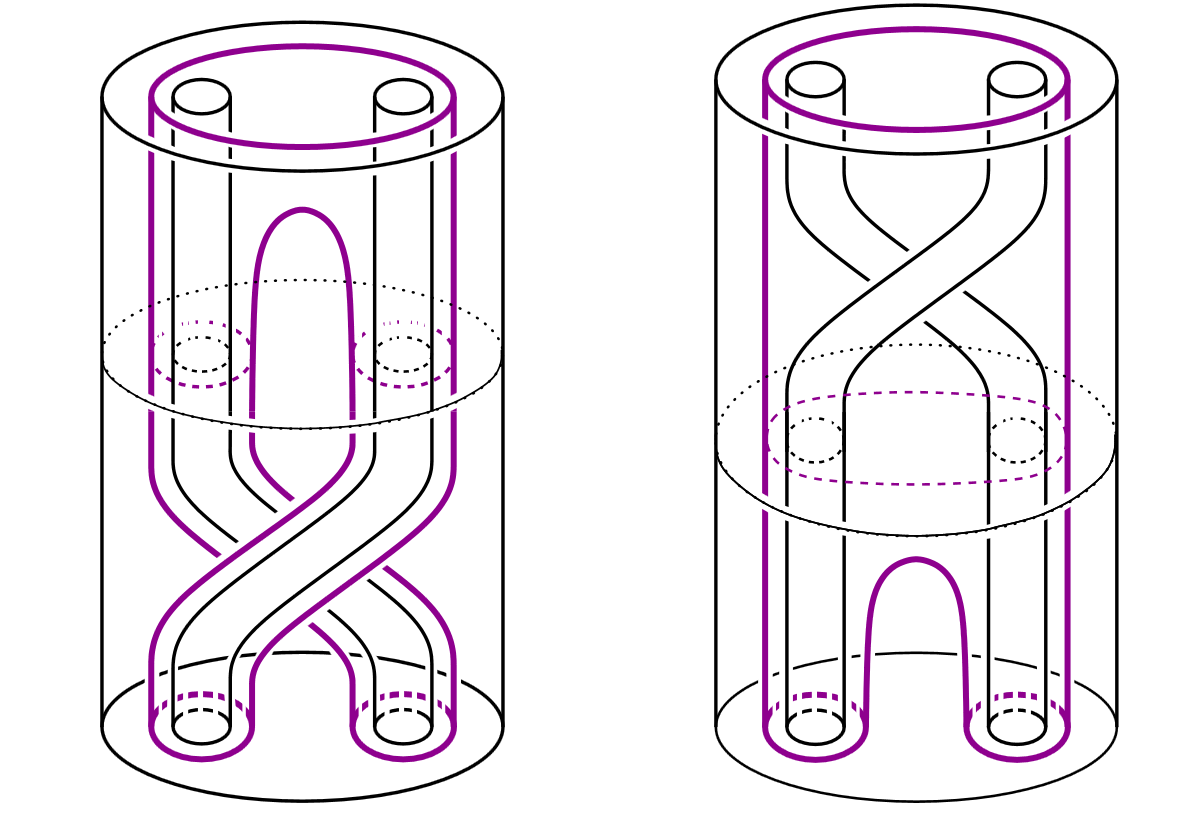}
\end{center}

For details we refer to 
\cite{FPSV14}.  We repeat their very interesting question  
linking this natural functor from the TFT construction to the ENOM functor, 
which they solve in the case $\Vect_G$ for $G$ abelian by explicit calculation 
using the bundle construction:
\begin{question}
  Does the assignment of the functor 
  $Z(\Sigma_\text{Transm}):\zcat(\cat)\to\zcat(\dcat)$ to an exact invertible 
  bimodule category ${_\cat}\mcat_\dcat$ coincide with ENOM functor?
\end{question}

The results of the present article give many new families of examples for such 
situations. The final hope is, that there are three types of defects and every 
defect can be written as a product. This would also open the possibility of 
checking the previous question explicitly for the given subgroups.\\

The TFT approach is also a reason for insisting in the formulation of 
exact invertible $\cat$-$\dcat$-bimodule categories with $\cat\neq\dcat$: As we 
saw, for quantum groups many of the interesting examples appear between 
different categories - an effect that is present (but rare) for group examples, 
see Example \ref{exm_exoticReflection}. From a physics perspective, it is very 
natural to assign different categories to different ``phases regions'' i.e. 
connected regions separated by defects.\\

\subsection{Outlook: Group-theoretic extensions}~\\

\noindent
By \cite{ENO09} group-theoretic extensions
$$\dcat=\bigoplus_{t\in\Sigma} \dcat_t\qquad \dcat_1=\cat$$
of the category $H\md\mod$ by $\Vect_\Sigma$ are 
associated to homomorphisms $\psi:\Sigma\to \BrPic(\cat)$ (plus 
additional coherence data we omit here) with $\dcat_t=\psi(t)$ a 
$\cat$-$\cat$-bimodule category.\\

We finally sketch briefly what the result is for $\cat=H\md\mod$ when $\psi$ 
lands in our three subgroups $\bcat,\ecat,\langle\pcat\rangle$ in 
$\BrPic(H\md\mod)$. The idea is that there are essentially three types of 
generic group-theoretic extensions associated to the three subgroups:\\

Let $\psi:\Sigma\to \bcat=\Ind(\AutMon(H\md\mod))$. This is the trivial case 
considered by several authors: All the bimodule categories are 
$\dcat_t={_{F_t}}H\md\mod$ so $\dcat=\Vect_\Sigma\boxtimes \cat$, while 
$\Sigma\to\AutMon(\H\md\mod)$ gives a categorical action and accordingly is the 
tensor product defined.

\begin{example}
 Take the case $\vcat$ i.e. let $v\in\AutHopf(H)$ of order $n$ and let 
$\Sigma=\langle v\rangle$ and $\psi$ just the identity. Then the associated 
category is
$$\dcat=\Vect_\Sigma\boxtimes H\md\mod
=\bigoplus_{i=0}^{n-1} H\md\mod$$
with a tensor product $X_i\otimes Y_j=(X\otimes v^i(Y))_{i+j}$. Hence $\dcat$ 
should be the representations of the cosmash product Hopf algebra 
$\ZZ_n\ltimes H$, with $\ZZ_n$-coaction on $H$ given by $v$, which is as an 
algebra just $\ZZ_n\otimes H$.
\end{example}

Let $\psi:\Sigma\to \ecat=\Ind(\AutMon(H^*\md\mod))=\Bigal(H)$. Then 
$\dcat=\tilde{H}\md\mod$ where the new Hopf algebra is as an algebra
 $$\tilde{H}=\bigoplus_{t\in\Sigma} R_t$$
with Bigalois objects $R_t$. This type of extensions has been considered in the 
first authors work \cite{Len12}, in particular in its application to construct 
new Nichols algebras.

\begin{example}
Let $\Sigma^*\to G\to \Gamma$ a central extension of groups, then associated 
one has a $2$-cocycle in $Z^2(\Gamma,\Sigma^*)$ and hence a homomorphism 
$\phi:\Sigma\to Z^2(\Gamma,\CC^\times)$. We viewing the target as the subgroup 
of $\bcat$ for $H=\CC[\Gamma]$. Then our construction returns bimodule 
categories $\dcat_t=R_t\md\mod$ for Bigalois objects being twisted group 
rings $R_t=\CC_{\phi(t)}[\Gamma]$ and overall we get 
$$\tilde{H}=\CC[G]
\qquad \dcat=\Rep(G) =\bigoplus_{t\in\Sigma} \CC_{\phi(t)}[\Gamma] 
\md\mod $$
For example $\CC[\DD_4]=\CC[\ZZ_2^2]\oplus \CC_\sigma[\ZZ_2^2]$ and 
$\dcat=\Rep(\DD_4)$ is a $\ZZ_2$-extension of $\Rep(\ZZ_2^2)$.  
\end{example}
\begin{example}[\cite{Len12}]
Let $\phi:\Sigma\to Z^2(\Gamma,\CC^\times)$ as above and $\Nic(M)$ a Nichols 
algebra over $\Gamma$, and assume we are given a so-called twisted symmetry 
action of $\Sigma$ on $\Nic(M)$. Then this data gives rise to a homomorphism 
$$R_t:\Sigma\to \Bigal(\Nic(M)\rtimes \CC[\Gamma])$$
and our construction returns $$\tilde{H}=\Nic(\tilde{M})\rtimes \CC[G]
\qquad \dcat=\tilde{H}\md\mod =\bigoplus_{t\in\Sigma} R_t\md\mod$$
where $\Nic(\tilde{M})$ is a Nichols algebra over the centrally extended group 
$G$.
\end{example}

Let $\Sigma=\ZZ_2$ and $\psi(g)=r$ a partial dualization on a semidirect 
product decomposition $H=K\rtimes A$ with $K$-self-dual. Then again 
we obtain a group-theoretical extension
$$\dcat=\dcat_1\oplus \dcat_r= H\md\mod \times A\md\mod$$
\begin{example}[\cite{ENO09} Sec. 9.2]
 Let $A=1$, e.g. for $H=K=\CC[G]$ with $G$ abelian. Then $\dcat$ is a 
Tambara-Yamigami category with $\dcat_1=\cat$ pointed and $\dcat_r$ consisting 
of a unique simple object. 
\end{example}
\begin{question}
  What are the category extension associated to $H=U_q(\g)^+$ and the 
homomorphism $\phi:W\to \BrPic(H)$ with $W$ the Weyl group generated by all 
reflections $r_i$? (or say a cyclic subgroup generated by a single element $w$) 
\end{question}

\noindent{\sc Acknowledgments}: 
We are grateful to C. Schweigert for many helpful
discussions. The authors are partially supported by the DFG Priority 
Program SPP 1388 ``Representation Theory'' and the Research Training Group 1670 
``Mathematics Inspired by String Theory and QFT''. 
S.L. is currently on a research stay supported by DAAD PRIME, funded by BMBF and 
EU Marie Curie Action.

\end{document}